\documentclass[a4paper,11pt]{article}
\usepackage{amsmath,amsthm}
\usepackage{authblk}
\usepackage{booktabs}
\usepackage{cancel}
\usepackage{colortbl}
\usepackage{comment}
\usepackage[hmargin={26mm,26mm},vmargin={30mm,35mm}]{geometry}
\usepackage{ifthen}
\usepackage{hyperref}
\usepackage{multirow}
\usepackage{nicefrac}
\usepackage{paralist}
\usepackage{tikz-cd}
\usepackage{xcolor}

\usepackage[color,notref,notcite,final]{showkeys}

\usepackage{makeidx}
\usepackage{graphicx}

\usepackage{newtxtext}
\usepackage{newtxmath}

\usepackage{accents}

\newcommand{\email}[1]{\href{mailto:#1}{#1}}

\numberwithin{equation}{section}

\newtheorem{theorem}{Theorem}
\newtheorem{proposition}[theorem]{Proposition}
\newtheorem{lemma}[theorem]{Lemma}

\theoremstyle{remark}
\newtheorem{remark}[theorem]{Remark}
\theoremstyle{definition}

\newcommand{\st}{\,:\,}
\newcommand{\Real}{\mathbb{R}}

\DeclareRobustCommand{\bvec}[1]{\boldsymbol{#1}}
\newcommand{\uvec}[1]{\underline{\bvec{#1}}}
\newcommand{\cvec}[1]{\bvec{\mathcal{#1}}}

\newcommand{\ud}{\,\mathrm{d}}

\newcommand{\rotation}[1]{\varrho_{#1}}

\DeclareMathOperator{\GRAD}{\bf grad}
\DeclareMathOperator{\CURL}{\bf curl}
\DeclareMathOperator{\DIV}{div}
\DeclareMathOperator{\ROT}{rot}
\DeclareMathOperator{\VROT}{\bf rot}

\newcommand{\Hcurl}[1]{\bvec{H}(\CURL;#1)}
\newcommand{\Hrot}[1]{\bvec{H}(\ROT;#1)}
\newcommand{\Hdiv}[1]{\bvec{H}(\DIV;#1)}

\newcommand{\uHgrad}[1][T]{\underline{X}_{\GRAD,#1}^k}
\newcommand{\uHrot}{\underline{\bvec{X}}_{\ROT,F}^k}
\newcommand{\uHcurl}[1][T]{\underline{\bvec{X}}_{\CURL,#1}^k}
\newcommand{\uHdiv}[1][T]{\underline{\bvec{X}}_{\DIV,#1}^k}

\newcommand{\uIgrad}[1][T]{\underline{I}_{\GRAD,#1}^k}
\newcommand{\uIrot}[1][F]{\uvec{I}_{\ROT,#1}^k}
\newcommand{\uhIrot}[1][F]{\accentset{\bullet}{\uvec{I}}_{\ROT,#1}^k}
\newcommand{\uIcurl}[1][T]{\uvec{I}_{\CURL,#1}^k}
\newcommand{\uIdiv}[1][T]{\uvec{I}_{\DIV,#1}^{k}}

\newcommand{\lproj}[2]{\pi_{\mathcal{P},#2}^{#1}}
\newcommand{\vlproj}[2]{\boldsymbol{\pi}_{\cvec{P},#2}^{#1}}

\newcommand{\Rproj}[2]{\bvec{\pi}_{\cvec{R},#2}^{#1}}
\newcommand{\ROproj}[2]{\bvec{\pi}_{\cvec{R},#2}^{\perp,#1}}

\newcommand{\Gproj}[2]{\bvec{\pi}_{\cvec{G},#2}^{#1}}
\newcommand{\GOproj}[2]{\bvec{\pi}_{\cvec{G},#2}^{\perp,#1}}

\newcommand{\uGT}[1][k]{\uvec{G}_T^{#1}}

\newcommand{\uGFh}[1][k]{\accentset{\bullet}{\uvec{G}}_F^{#1}}
\newcommand{\uGF}[1][k]{\uvec{G}_F^{#1}}
\newcommand{\GOT}{\bvec{G}_T^{\perp,k}}

\newcommand{\GRF}{\bvec{G}_{\cvec{R},F}^{k-1}}
\newcommand{\GROF}{\bvec{G}_{\cvec{R},F}^{\perp,k}}
\newcommand{\GRT}{\bvec{G}_{\cvec{R},T}^{k-1}}
\newcommand{\GROT}{\bvec{G}_{\cvec{R},T}^{\perp,k}}

\newcommand{\uCT}[1][]{\uvec{C}_T^{k\ifthenelse{\equal{#1}{}}{}{,#1}}}

\newcommand{\CGT}[1][k-1]{\bvec{C}_{\cvec{G},T}^{#1}}
\newcommand{\CGOT}[1][k]{\bvec{C}_{\cvec{G},T}^{\perp,#1}}

\newcommand{\cGpF}{G_{\partial F}^k}
\newcommand{\cGE}{G_E^k}
\newcommand{\cGF}{\accentset{\bullet}{\bvec{G}}_F^k}
\newcommand{\cGT}{\accentset{\bullet}{\bvec{G}}_T^k}

\newcommand{\cCF}{C_F^k}
\newcommand{\hcCF}{\accentset{\bullet}{C}_F^k}
\newcommand{\cCT}{\accentset{\bullet}{\bvec{C}}_T^k}

\newcommand{\cDT}{D_T^k}

\newcommand{\trt}{\gamma_{\rm t}}

\newcommand{\trF}{\gamma_F^{k+1}}
\newcommand{\trFtilde}{\tilde{\gamma}_F^{k+1}}
\newcommand{\trFt}{\bvec{\gamma}_{{\rm t},F}^k}

\newcommand{\pgrad}{P_{\GRAD,T}^{k+1}}
\newcommand{\pcurl}{\bvec{P}_{\CURL,T}^k}
\newcommand{\pdiv}{\bvec{P}_{\DIV,T}^k}

\newcommand{\FT}{\mathcal{F}_T}
\newcommand{\ET}{\mathcal{E}_T}
\newcommand{\EF}{\mathcal{E}_F}
\newcommand{\FE}{\mathcal{F}_E}
\newcommand{\VT}{\mathcal{V}_T}
\newcommand{\VF}{\mathcal{V}_F}

\newcommand{\normal}{\bvec{n}}
\newcommand{\tangent}{\bvec{t}}

\newcommand{\Poly}[2][]{\mathcal{P}_{#1}^{#2}}
\newcommand{\Qoly}[2][]{\mathcal{Q}_{#1}^{#2}}
\newcommand{\bQoly}[1]{\boldsymbol{\mathcal{Q}}^{#1}}

\newcommand{\Roly}[1]{\boldsymbol{\mathcal{R}}^{#1}}
\newcommand{\Goly}[1]{\boldsymbol{\mathcal{G}}^{#1}}

\newcommand{\NE}[1]{\boldsymbol{\mathcal{N}}^{#1}}
\newcommand{\RT}[1]{\boldsymbol{\mathcal{RT}}^{#1}}

\newcommand{\norm}[2][]{\|#2\|_{#1}}

\DeclareMathOperator{\Ker}{Ker}
\DeclareMathOperator{\Image}{Im}

\DeclareMathOperator{\card}{card}

\begin{document}

\title{Fully discrete polynomial de Rham sequences of arbitrary degree on polygons and polyhedra}
\author[1]{Daniele A. Di Pietro}
\author[2]{J\'er\^ome Droniou}
\author[3]{Francesca Rapetti}
\affil[1]{IMAG, Univ Montpellier, CNRS, Montpellier, France, \email{daniele.di-pietro@umontpellier.fr}}
\affil[2]{School of Mathematics, Monash University, Melbourne, Australia, \email{jerome.droniou@monash.edu}}
\affil[3]{Laboratoire J. A. Dieudonn\'e, Univ. C\^{o}te d'Azur, Nice, France \email{frapetti@univ-cotedazur.fr}}

\maketitle

\begin{abstract}
  In this work, merging ideas from compatible discretisations and polyhedral methods, we construct novel fully discrete polynomial de Rham sequences of arbitrary degree on polygons and polyhedra.
  The spaces and operators that appear in these sequences are directly amenable to computer implementation.
  Besides proving exactness, we show that the usual three-dimensional sequence of trimmed Finite Element spaces forms, through appropriate interpolation operators, a commutative diagram with our sequence, which ensures suitable approximation properties.
  A discussion on reconstructions of potentials and discrete $L^2$-products completes the exposition.
\end{abstract}

\section{Introduction}

In this paper we construct novel fully discrete polynomial de Rham sequences of arbitrary degree on polygons and polyhedra.
By fully discrete, we mean that both the spaces and vector operators that appear in the sequence are directly amenable to computer implementation.
This construction can be used to design stable schemes for complex problems, such as the ones encountered in computational electromagnetism, on meshes far more general than classical Finite Element methods.
In particular, the supported meshes can contain general polyhedral elements and nonmatching interfaces, enabling advanced computational strategies such as nonconforming mesh refinement, mesh coarsening, etc.
\medskip

The ideas underlying this work result from the confluence of two streams of research that have gathered an enormous amount of attention in the numerical community over the last years: compatible discretisations and polytopal methods.
\smallskip

Compatible discretisations aim at preserving structural features of the continuous model at the discrete level.
Such features are instrumental to obtaining the stability and consistency properties required for convergence when non-trivial operators and domains are considered, as is the case in computational electromagnetism (see Section \ref{sec:motivation}).
The origins of compatible discretisations can be tracked back to, e.g., \cite{Weil:52,Whitney:57,Dodziuk:76} for the mathematical community and \cite{Yee:66,Tonti:75,Tonti:75*1,Bossavit:88,Frankel:97} for the electromagnetic one; see also the survey paper \cite{Hiptmair:02}.
The importance of concepts from differential geometry and algebraic topology in the formulation of compatible discretisations is nowadays widely recognised; see, e.g., \cite{Mattiussi:97,Bossavit:98,Bochev.Hyman:06,Arnold.Falk.ea:06,Desbrun.Kanso.ea:08,Gerritsma:11,Teixeira:13,Arnold:18}.
As a matter of fact, by reformulating the continuous problems in terms of differential forms, one gets some indications on the design of suitable Finite Element (FE) discretisations.
Specifically:
\begin{compactenum}[(i)]
\item the choice of the degrees of freedom (DOFs) should reflect the nature (and global regularity properties) of the fields they represent;
\item the discrete spaces and operators should form an exact sequence;
\item the operator that maps DOFs to forms commutes with the continuous/discrete exterior derivative.
\end{compactenum}
An important issue when generating high-order discrete de Rham sequences in the FE spirit lies in the choice of the bases and of the DOFs.
A reformulation of classical moments that underlines their geometrical aspects has been recently proposed in \cite{Bonazzoli.Rapetti:17}.
This reformulation has been made possible by the precursor works \cite{Rapetti.Bossavit:09,Christiansen.Rapetti:16}, where new DOFs in terms of weights of forms on small chains have been proposed.
These weights have shed new light on the high-order approximations originally proposed by N\'ed\'elec \cite{Nedelec:80} confirming, also for the high-order version, the tight relation with Whitney forms \cite{Whitney:57}; see \cite{Bossavit:88,Bossavit:01,Bossavit:02} for the low-order case.
A generalisation to the high-order case of the relations between the moments of a field and those of its potential has been proposed in \cite{Alonso-Rodriguez.Rapetti:18}.

The extension of the FE approach to more general meshes is, however, not straightforward.
The main reason is that, in order to construct a conforming FE discretisation, one has to devise discrete spaces that, through the single-valuedness of DOFs at element boundaries, satisfy suitable global continuity requirements.
Recent efforts in this direction have been made in, e.g., \cite{Gillette.Rand.ea:16,Chen.Wang:17} (see also references therein), focusing mainly on the lowest-order case and with some limitations on the element shapes in three dimensions.
\medskip

The problem of devising discretisation methods that support more general meshes than classical FE (including, e.g., polytopal elements and nonmatching interfaces) has been recently tackled with great impetus by the numerical community.
As pointed out above, supporting general meshes paves the way to computational strategies that are typically not accessible to traditional conforming FE (nonconforming mesh refinement, mesh coarsening, seamless handling of fractures and microstructures, etc.).
We will focus here only on those developments that bear relations to the approach proposed in this work, and refer the reader to the preface of \cite{Di-Pietro.Droniou:20} for a literature review of broader scope.

Let us start with lowest-order methods that fall within the category of compatible discretisations.
Mimetic Finite Differences (MFD) are derived by mimicking the Stokes theorem to formulate discrete counterparts of differential operators and $L^2$-products \cite{Beirao-da-Veiga.Lipnikov.ea:14}.
Their extension to polytopal meshes has been first carried out in \cite{Kuznetsov.Lipnikov.ea:04,Lipnikov.Shashkov.ea:06}, then analysed in \cite{Brezzi.Lipnikov.ea:05,Brezzi.Buffa.ea:09}; see also \cite{Droniou.Eymard.ea:10} for a link with the Mixed Hybrid Finite Volume (MHFV) methods of \cite{Droniou.Eymard:06,Eymard.Gallouet.ea:10} and \cite[Section 2.5]{Di-Pietro.Ern.ea:14} along with \cite[Section 3.5]{Di-Pietro.Ern:17} for links with Hybrid High-Order (HHO) methods.
In the Discrete Geometric Approach (DGA), originally introduced in \cite{Codecasa.Specogna.ea:07} and extended to polyhedral meshes in \cite{Codecasa.Specogna.ea:09,Codecasa.Specogna.ea:10}, as well as in Compatible Discrete Operators \cite{Bonelle.Ern:14,Bonelle.Di-Pietro.ea:15}, formal links with the continuous operators are expressed in terms of Tonti diagrams \cite{Tonti:75,Tonti:13}.
Similarly to the approach pursued here, MFD, DGA, and CDO methods work on discrete unknowns and rely on discrete counterparts of the vector operators.
Contrary to the present work, however, they are typically limited to the lowest-order and their analysis often relies on an interplay of functional and algebraic arguments that is not required in our presentation.

The development of high-order schemes is more recent.
A high-order approach with structure-preserving features is provided by the Virtual Element Method (VEM); see \cite{Beirao-da-Veiga.Brezzi.ea:13}.
VEM can be described as FE methods where explicit expressions for the basis functions are not available at each point; hence the term ``virtual'' in reference to the function space they span.
The DOFs are selected so as to enable the computation of polynomial projections of virtual functions and vector operators, which are used in turn to formulate local contributions involving consistency and stabilisation terms.
An exact de Rham sequence of virtual spaces on polyhedra has been recently proposed in \cite{Beirao-da-Veiga.Brezzi.ea:16}, with polynomial degrees decreasing by one at each application of the exterior derivative (other virtual sequences are presently under investigation \cite{Brezzi.Marini:19}, see also the related works \cite{Beirao-da-Veiga.Brezzi.ea:18,Beirao-da-Veiga.Brezzi.ea:18*1} concerning applications to magnetostatics).

Owing to the variational crime committed when taking projections on polynomial spaces, the exactness of the virtual sequence cannot be directly exploited to obtain stable numerical approximations.
The approach proposed in this work, inspired by the HHO literature \cite{Di-Pietro.Ern:15,Di-Pietro.Ern:17,Di-Pietro.Droniou:20}, aims at establishing the exactness property for \emph{fully discrete} de Rham sequences, i.e., sequences involving spaces of (polynomial) discrete unknowns and discrete counterparts of vector operators acting thereon.
The starting point is to identify arbitrary-order reconstructions of vector operators in full polynomial spaces.
These reconstructions allow one to identify appropriate sets of discrete unknowns, which play the role of DOFs in standard FE (or VEM). 
To ensure the compatibility with the choice of unknowns, each discrete vector operator is attached to the appropriate geometric entities: in three space dimensions, the discrete gradient has components on the edges, faces, and inside the polyhedron; the discrete curl has components at faces and inside the polyhedron; the discrete divergence has only one component inside the polyhedron.
The full reconstructions of vector operators cannot be directly used to form an exact sequence, but their study permits to identify the modifications required to recover exactness.
Specifically, an exact sequence is obtained by restricting the domains/co-domains of the operators in the middle of the sequence and taking the $L^2$-orthogonal projections of the full vector operator reconstructions on these spaces.
Crucially, the proof of exactness relies on purely discrete arguments, that do not involve spaces of non-polynomial functions.
The sequence we focus on is constructed so that all the spaces involved have the same polynomial degree and so that, through appropriate interpolation operators, it forms commutative diagrams with the usual sequence of trimmed FE spaces, which warrants suitable approximation properties.
To complete the exposition, we also show how to reconstruct consistent potentials in each space and write discrete and consistent counterparts of $L^2$-products based on the latter.
The focus of this paper is on the development of the exact discrete sequence; applications are postponed to future works.

\medskip
The rest of this work is organised as follows:
in Section \ref{sec:motivation} we provide a motivation for the present work by pinpointing the role of the de Rham sequence in proving the well-posedness of the electrostatic problem in mixed formulation;
in Section \ref{sec:basic.tools} we introduce the basic tools and notations;
in Sections \ref{sec:2d.sequence} and \ref{sec:3d.sequence} we construct fully discrete arbitrary-order exact sequences in two and three space dimensions, respectively.


\section{Motivation}\label{sec:motivation}

We start with a brief discussion that motivates the need for discretisation techniques that reproduce the exactness of the continuous de Rham complex and provide suitable $L^2$-inner products on discrete spaces. We consider the $\CURL \CURL$ magnetostatic problem, set on a bounded connected domain $\Omega\subset \Real^3$ without voids (that is, $\Omega$ has a zero second Betti number):
\begin{subequations}\label{eq:strong}
  \begin{alignat}{2}
    \bvec{\sigma} - \CURL\bvec{u} &= \bvec{0} &\qquad& \text{in $\Omega$},
    \\
    \CURL\bvec{\sigma} &= \bvec{f} &\qquad& \text{in $\Omega$},
    \\
    \DIV\bvec{u} &= 0 &\qquad& \text{in $\Omega$},
    \\
    \bvec{u}\times\normal &= \bvec{g} &\qquad& \text{on $\partial\Omega$},
  \end{alignat}
\end{subequations}
where $\bvec{f}\in\CURL\left(\Hcurl{\Omega}\right)$ and $\bvec{g}\in L^2(\partial\Omega)^3$.
As discussed in  \cite[Section 4.5.3]{Arnold:18}, and accounting for the assumptions on $\Omega$ (which imply that its space of 2-harmonic forms is trivial), the weak formulation of this problem is: Find $(\bvec{\sigma},\bvec{u})\in \Hcurl{\Omega}\times\Hdiv{\Omega}$ such that
\begin{equation}\label{eq:weak}
  A((\bvec{\sigma},\bvec{u}),(\bvec{\tau},\bvec{v}))
  =
  \langle\bvec{g},\trt\bvec{\tau}\rangle_{\partial\Omega}
  +\int_\Omega\bvec{f}\cdot\bvec{v}\qquad\forall (\bvec{\tau},\bvec{v})\in \Hcurl{\Omega}\times\Hdiv{\Omega},
\end{equation}
where $\trt$ denotes the tangential trace operator on $\partial\Omega$, $\langle\cdot,\cdot\rangle_{\partial\Omega}$ the duality pairing betweeen $H^{\frac12}(\partial\Omega)$ and $H^{-\frac12}(\partial\Omega)$, while the bilinear form $A: [\Hcurl{\Omega}\times\Hdiv{\Omega}]^2\to\Real$ is such that
\begin{equation}\label{eq:def.A}
  A((\bvec{\sigma},\bvec{u}),(\bvec{\tau},\bvec{v}))
  \coloneq\int_\Omega\bvec{\sigma}\cdot\bvec{\tau}-\int_\Omega\bvec{u}\cdot\CURL\bvec{\tau}+\int_\Omega\bvec{v}\cdot\CURL\bvec{\sigma}+\int_\Omega\DIV\bvec{u}\DIV\bvec{v}.
\end{equation}
The well-posedness of problem \eqref{eq:weak} classically \emph{requires} the bilinear form $A$ to satisfy the following inf-sup condition.

\begin{lemma}[Inf-sup condition for $A$] If $\Omega$ is bounded, connected and without holes, it holds, with $C>0$ only depending on $\Omega$,
\begin{multline}\label{eq:inf.sup.A}
  \sup_{(\bvec{\tau},\bvec{v})\in \Hcurl{\Omega}\times\Hdiv{\Omega}\backslash\{(\bvec{0},\bvec{0})\}}\frac{A((\bvec{\sigma},\bvec{u}),(\bvec{\tau},\bvec{v}))}{\norm[\Hcurl{\Omega}\times \Hdiv{\Omega}]{(\bvec{\tau},\bvec{v})}}
  \ge C\norm[\Hcurl{\Omega}\times\Hdiv{\Omega}]{(\bvec{\sigma},\bvec{u})}\\
  \forall (\bvec{\sigma},\bvec{u})\in \Hcurl{\Omega}\times\Hdiv{\Omega},
\end{multline}
where $\Hcurl{\Omega}$ and $\Hdiv{\Omega}$ are equipped with their standard norms.
\end{lemma}

\begin{proof}
Let $\mathcal S$ denote the left-hand side of \eqref{eq:inf.sup.A} and let us start by choosing $(\bvec{\tau},\bvec{v})=(\bvec{\sigma},\bvec{u})$ in \eqref{eq:def.A}, which readily gives
\begin{equation}\label{eq:stab.cont.1}
\mathcal S \ge \frac{\norm[L^2(\Omega)^3]{\bvec{\sigma}}^2+\norm[L^2(\Omega)]{\DIV \bvec{u}}^2}{\norm[\Hcurl{\Omega}\times\Hdiv{\Omega}]{(\bvec{\sigma},\bvec{u})}}
\end{equation}
We then choose, in \eqref{eq:def.A}, $\bvec{\tau}=\bvec{0}$ and $\bvec{v}=\CURL\bvec{\sigma}$, the latter belonging to $\Hdiv{\Omega}$ since $\Image\CURL \subset \Ker\DIV$.
This gives
$\mathcal S \ge \norm[L^2(\Omega)^3]{\CURL\bvec{\sigma}}$ which, combined with \eqref{eq:stab.cont.1}, leads to
\begin{equation}\label{eq:stab.cont.2}
\mathcal S\ge C\frac{\norm[\Hcurl{\Omega}]{\bvec{\sigma}}^2+\norm[L^2(\Omega)]{\DIV \bvec{u}}^2}{\norm[\Hcurl{\Omega}\times\Hdiv{\Omega}]{(\bvec{\sigma},\bvec{u})}}.
\end{equation}
Here and in the rest of the proof, $C$ is a generic strictly positive constant that does not depend on $(\bvec{\sigma},\bvec{u})$. To conclude the proof of \eqref{eq:inf.sup.A}, it remains to estimate the $L^2$-norm of $\bvec{u}$.
We split this function into
\[
\bvec{u}=\bvec{u}^\star+\bvec{u}^\perp\in \Ker\DIV \oplus (\Ker\DIV)^\perp=L^2(\Omega)^3,
\]
where the orthogonal is taken with respect to the $L^2$-inner product.
A consequence of the exactness relation
\begin{equation}\label{eq:exact.div.L2}
\Image\DIV=L^2(\Omega)
\end{equation}
of the de Rham sequence is that $\DIV:(\Ker\DIV)^\perp\to L^2(\Omega)$ is an isomorphism, and therefore has a continuous inverse mapping. This implies
\begin{equation}\label{eq:stab.cont.3}
\norm[L^2(\Omega)^3]{\bvec{u}^\perp}^2\le C \norm[L^2(\Omega)]{\DIV\bvec{u}^\perp}^2=C\norm[L^2(\Omega)]{\DIV\bvec{u}}^2\le
C\mathcal S \norm[\Hcurl{\Omega}\times\Hdiv{\Omega}]{(\bvec{\sigma},\bvec{u})},
\end{equation}
where the equality follows from $\DIV\bvec{u}=\DIV(\bvec{u}^\star+\bvec{u}^\perp)=\DIV\bvec{u}^\perp$ and \eqref{eq:stab.cont.2} was used in the last inequality. To estimate the $L^2$-norm of $\bvec{u}^\star$, we use the exactness relation
\begin{equation}\label{eq:exact.curl.div}
\Image\CURL=\Ker\DIV,
\end{equation}
which is valid thanks to the assumption on $\Omega$.
This relation entails that $\CURL:(\Ker\CURL)^\perp\to \Image\CURL=\Ker\DIV$ is an isomorphism, with a continuous inverse mapping; since $\bvec{u}^\star\in\Ker\DIV$, we can therefore find $\bvec{\tau}\in(\Ker\CURL)^\perp$ such that $\CURL\bvec{\tau}=-\bvec{u}^\star$ and $\norm[L^2(\Omega)^3]{\bvec{\tau}}\le C\norm[L^2(\Omega)^3]{\CURL\bvec{\tau}}=C\norm[L^2(\Omega)^3]{\bvec{u}^\star}$, which implies
\begin{equation}\label{eq:stab.cont.4}
\norm[\Hcurl{\Omega}]{\bvec{\tau}}\le C \norm[L^2(\Omega)^3]{\bvec{u}^\star}.
\end{equation}
Plugging this $\bvec{\tau}$ with $\bvec{v}=\bvec{0}$ into the supremum defining $\mathcal S$ and recalling the definition of $A$, we therefore obtain
\[
C\norm[L^2(\Omega)^3]{\bvec{u}^\star}\mathcal S \ge  \norm[\Hcurl{\Omega}]{\bvec{\tau}} \mathcal S
\ge \int_\Omega \bvec{\sigma}\cdot\bvec{\tau}+\int_\Omega (\bvec{u}^\star+\bvec{u}^\perp)\cdot\bvec{u}^\star.
\]
Cauchy--Schwarz and Young inequalities together with \eqref{eq:stab.cont.2}, \eqref{eq:stab.cont.3}, and \eqref{eq:stab.cont.4} lead to
\[
\mathcal S \norm[\Hcurl{\Omega}\times\Hdiv{\Omega}]{(\bvec{\sigma},\bvec{u})}\ge C\norm[L^2(\Omega)^3]{\bvec{u}^\star}
\]
which, combined with \eqref{eq:stab.cont.2} and \eqref{eq:stab.cont.3}, concludes the proof of \eqref{eq:inf.sup.A}.
\end{proof}

This proof shows how crucial working within a de Rham sequence characterised by the exactness relations \eqref{eq:exact.div.L2} and \eqref{eq:exact.curl.div} is to establish the well-posedness of \eqref{eq:weak}. Obtaining a well-posed numerical scheme for this problem requires to reproduce  this framework at the discrete level; the companion paper \cite{Di-Pietro.Droniou:20*1} shows how to carry out this reproduction, based on the local spaces and operators constructed below.
Additionally, $L^2$-like inner products on the discrete version of $\Hcurl{\Omega}$ and $\Hdiv{\Omega}$ must be defined, from the scheme's DOFs, to obtain the discrete version of the bilinear form $A$; to achieve high accuracy of this approximation, these inner products should fulfil polynomial consistency properties -- that is, the discrete inner products applied to DOFs corresponding to certain polynomials of a suitable degree must return the same values as the continuous $L^2$-inner products applied to these polynomials. Polynomial accuracy can be achieved by taking the continuous $L^2$-inner products of potentials, reconstructed in a polynomially consistent way from the DOFs.


\section{Basic tools and notation}\label{sec:basic.tools}

\subsection{Polyhedra and polygons}

A polytope of $\Real^d$, $d\ge 1$, is a connected set that is the interior of a finite union of simplices.
Our focus will be here on polytopes in dimension $d=2$ (polygons) and $d=3$ (polyhedra).
We assume that these polytopes are simply connected and have connected boundaries that are Lipschitz-continuous (that is, each polytope can locally be represented as epigraphs of Lipschitz-continuous functions). In practical applications, the polytopes are elements from the computational mesh, resulting either from mesh generation or from mesh coarsening/agglomeration \cite{Bassi.Botti.ea:12}. The simple connectedness and Lipschitz boundary assumptions therefore do not entail any unreasonable practical restriction on their geometry.

Given a polyhedron $T\subset\Real^3$, we denote by $\FT$ the set of planar polygonal faces that lie on the boundary of $T$.
For all $F\in\FT$, an orientation is set by prescribing a unit normal vector $\normal_F$, and we denote by $\omega_{TF}\in\{-1,1\}$ the orientation of $F$ relative to $T$, that is, $\omega_{TF}=1$ if $\normal_F$ points out of $T$, $-1$ otherwise.
With this choice, $\omega_{TF}\normal_F$ is the unit vector normal to $F$ that points out of $T$.
Similarly, for a polygon $F$, we denote by $\EF$ the set of edges that lie on the boundary $\partial F$ of $F$.
Notice that, throughout the paper, polygons are tacitly regarded as immersed in $\Real^3$ whenever needed.
For all $E\in\EF$, an orientation is set by prescribing the unit tangent vector $\tangent_E$.
The boundary of $F$ is oriented counter-clockwise with respect to $\normal_F$, and we denote by $\omega_{FE}\in\{-1,1\}$ the orientation of $\tangent_E$ opposite to $\partial F$: $\omega_{FE}=1$ if $\tangent_E$ points on $E$ in the opposite orientation to $\partial F$, $\omega_{FE}=-1$ otherwise.
The vertices $V_1$, $V_2$ of the edge $E$ have coordinates $\bvec{x}_{E,1}$, $\bvec{x}_{E,2}$ and are numbered so that $|E|\tangent_E=\bvec{x}_{E,2}-\bvec{x}_{E,1}$, where $|E|$ denotes the length of $E$.
For any polygon $F$ and any edge $E\in\EF$, we also denote by $\normal_{FE}$ the unit normal vector to $E$ lying in the plane of $F$ such that $(\tangent_E,\normal_{FE})$ form a system of right-handed coordinates in the plane of $F$, which means that the system of coordinates $(\tangent_E,\normal_{FE},\normal_F)$ is right-handed.
It can be checked that $\omega_{FE}\normal_{FE}$ is the normal to $E$, in the plane where $F$ lies, pointing out of $F$, and that, if $F_1,F_2$ are two faces of $T$ that share an edge $E$, it holds
\begin{equation}\label{eq:orientation}
  \omega_{TF_1}\omega_{F_1E}+\omega_{TF_2}\omega_{F_2E}=0.
\end{equation}
In what follows, we will also need the sets of edges and vertices of a polyhedron $T\subset\Real^3$, which we denote by $\ET$ and $\VT$, respectively, as well as the set $\partial^2 T\coloneq\bigcup_{E\in\ET}\overline{E}$.
The set of vertices of a polygon $F$ will be denoted by $\VF$.
The vector of coordinates of a generic vertex $V$ will be denoted by $\bvec{x}_V$.

\subsection{Polynomial spaces and vector operators}\label{sec:notation:polynomial.spaces.projectors}

For given integers $\ell\ge 0$ and $n\ge 0$, we denote by $\mathbb{P}_n^\ell$ the space of $n$-variate polynomials of total degree $\le\ell$, with the convention that $\mathbb{P}_0^\ell=\Real$ for any $\ell$ and $\mathbb{P}_n^{-1}=\{0\}$ for any $n$.
For $X$ polyhedron, polygon (immersed in $\Real^3$), or segment (again immersed in $\Real^3$), we denote by $\Poly{\ell}(X)$ the space spanned by the restriction to $X$ of functions in $\mathbb{P}_3^\ell$.
Denoting by $0\le n\le 3$ the dimension of $X$, $\Poly{\ell}(X)$ is isomorphic to $\mathbb{P}_n^\ell$ (the proof, quite simple, follows the ideas of \cite[Proposition 1.23]{Di-Pietro.Droniou:20}).
With a little abuse of notation, we denote both spaces with $\Poly{\ell}(X)$, and the exact meaning of this symbol should be inferred from the context.
We will also need the subspace $\Poly{0,\ell}(X)\coloneq\left\{q\in\Poly{\ell}(X)\st\int_X q=0\right\}$.
For any $X$ polygon or polyhedron, the $L^2$-orthogonal projector $\lproj{\ell}{X}:L^1(X)\to\Poly{\ell}(X)$ is such that, for any $q\in L^1(X)$,
\begin{equation}\label{eq:lproj}
  \int_X(\lproj{\ell}{X}q-q)r=0\qquad\forall r\in\Poly{\ell}(X).
\end{equation}
As a projector, $\lproj{\ell}{X}$ is polynomially consistent, that is, it maps any $r\in\Poly{\ell}(X)$ onto itself.
Optimal approximation properties for this projector have been proved in \cite{Di-Pietro.Droniou:17}; see also \cite{Di-Pietro.Droniou:17*1} for more general results on projectors on local polynomial spaces.
Denoting by $n$ the dimension of $X$, we also denote by $\vlproj{\ell}{X}:L^1(X)^n\to\Poly{\ell}(X)^n$ the vector version defined applying the projector component-wise.

Let $T$ be a polyhedron in $\Real^3$. We denote by $\Poly{\ell}(\FT)$ the space of functions $q:\partial T\to\Real$ such that $q_{|F}\in\Poly{\ell}(F)$ for all $F\in\FT$; an element $q\in\Poly{\ell}(\FT)$ is identified with the family $(q_{|F})_{F\in\FT}$ of its restrictions to the faces. The space $\Poly{\ell}(\ET)$ is defined similarly, replacing faces by edges of $T$. If $F$ is a polygon immersed in $\Real^3$, we define in the same way the space $\Poly{\ell}(\EF)$.
We also define the spaces $\Poly[c]{\ell}(\partial^2 T)\coloneq\Poly{\ell}(\ET)\cap C^0(\partial^2 T)$ and $\Poly[c]{\ell}(\partial F)\coloneq\Poly{\ell}(\EF)\cap C^0(\partial F)$ of functions that are continuous on the corresponding boundary and polynomial of degree $\le\ell$ on each edge of this boundary. It is easily checked that the following mapping is an isomorphism: 
\begin{equation}\label{eq:iso.Pc}
  \Poly[c]{\ell}(\partial F)\ni q\mapsto \big(
  (\lproj{\ell-1}{E}q)_{E\in\EF},(q(\bvec{x}_V))_{V\in\VF}
  \big)\in\left(
  \bigtimes_{E\in\EF}\Poly{\ell-1}(E)
  \right)\times\Real^{\VF}.
\end{equation}
A similar isomorphism can be constructed for $\Poly[c]{\ell}(\partial^2 T)$.

We respectively denote by $\GRAD_F$ and $\DIV_F$ the tangent gradient and divergence operators acting on smooth enough functions defined on $F\in\FT$.
Moreover, for any $r:F\to\Real$ smooth enough, we define the two-dimensional vector curl operator such that
\begin{equation}\label{eq:def:VROTF}
  \VROT_F r\coloneq \rotation{-\nicefrac\pi2}(\GRAD_F r),
\end{equation}
where $\rotation{-\nicefrac\pi2}$ is the rotation, in the oriented tangent space to $F$, of angle $-\frac\pi2$.
We will also need the two-dimensional scalar curl operator such that, for any $\bvec{v}:F\to\Real^2$ smooth enough,
\begin{equation}\label{eq:def:ROTF}
  \ROT_F\bvec{v}\coloneq\DIV_F (\rotation{-\nicefrac\pi2} \bvec{v}).
\end{equation}
We note, for future use, the following formulas linking volume and surface operators, which can be established selecting an orthonormal basis of $\Real^3$ in which $\normal_F$ is the third vector:
For any polyhedron $T\subset \Real^3$, any face $F\in\FT$ and any sufficiently smooth functions $\bvec{v}:T\to\Real^3$ and $r:T\to\Real$,
\begin{align}
\label{eq:gradT.rotF}
(\GRAD r)_{|F}\times\normal_F &= \VROT_F(r_{|F}),\\
\label{eq:curlT.divF.rotF}
(\CURL \bvec{v})_{|F}\cdot\normal_F &= \DIV_F(\bvec{v}_{|F}\times\normal_F)=
\ROT_F(\normal_F\times(\bvec{v}_{|F}\times\normal_F)).
\end{align}
Above, $\normal_F\times(\bvec{v}_{|F}\times\normal_F)$ is the orthogonal projection of $\bvec{v}_{|F}$ on the plane that contains $F$.
Notice that, here and in what follows, with a little abuse of notation, this and similar quantities are regarded as functions $F\to\Real^2$ whenever necessary.

For any integer $\ell\ge -1$, we define the following relevant subspaces of $\Poly{\ell}(F)^2$:
\[
\begin{alignedat}{3}
  \Goly{\ell}(F)&\coloneq\GRAD_F\Poly{\ell+1}(F),
  &\qquad&
  \Goly{\ell}(F)^\perp&\coloneq\text{$L^2$-orthogonal complement of $\Goly{\ell}(F)$ in $\Poly{\ell}(F)^2$},
  \\
  \Roly{\ell}(F)&\coloneq\VROT_F\Poly{\ell+1}(F),
  &\qquad&
  \Roly{\ell}(F)^\perp&\coloneq\text{$L^2$-orthogonal complement of $\Roly{\ell}(F)$ in $\Poly{\ell}(F)^2$}.
\end{alignedat}
\]
The corresponding $L^2$-orthogonal projectors are, with obvious notation, $\Gproj{\ell}{F}$, $\GOproj{\ell}{F}$, $\Rproj{\ell}{F}$, and $\ROproj{\ell}{F}$.

Similarly, given a polyhedron $T\subset\Real^3$, for any integer $\ell\ge -1$ we introduce the following subspaces of $\Poly{\ell}(T)^3$:
\[
\begin{alignedat}{3}
  \Goly{\ell}(T)&\coloneq\GRAD\Poly{\ell+1}(T),
  &\qquad&
  \Goly{\ell}(T)^\perp&\coloneq\text{$L^2$-orthogonal complement of $\Goly{\ell}(T)$ in $\Poly{\ell}(T)^3$},
  \\
  \Roly{\ell}(T)&\coloneq\CURL\Poly{\ell+1}(T),
  &\qquad&
  \Roly{\ell}(T)^\perp&\coloneq\text{$L^2$-orthogonal complement of $\Roly{\ell}(T)$ in $\Poly{\ell}(T)^3$},
\end{alignedat}
\]
The corresponding $L^2$-orthogonal projectors are $\Gproj{\ell}{T}$, $\GOproj{\ell}{T}$, $\Rproj{\ell}{T}$, and $\ROproj{\ell}{T}$.

For any polygon $F$, polyhedron $T$, and polynomial degree $\ell\ge 0$, the following mappings are isomorphisms:
\begin{alignat}{2}
  \VROT_F&:\Poly{0,\ell}(F)\xrightarrow{\cong}\Roly{\ell-1}(F)\,,\qquad &\GRAD&:\Poly{0,\ell}(T)\xrightarrow{\cong}\Goly{\ell-1}(T),
  \label{eq:iso:VROTF.GRAD}\\
  \DIV_F &:\Roly{\ell}(F)^\perp\xrightarrow{\cong} \Poly{\ell-1}(F)\,,\qquad&\DIV&:\Roly{\ell}(T)^\perp\xrightarrow{\cong}\Poly{\ell-1}(T),
  \label{eq:iso:DIV}\\
  &&\CURL &:\Goly{\ell}(T)^\perp\xrightarrow{\cong}\Roly{\ell-1}(T).
  \label{eq:iso:CURL}
\end{alignat}
The isomorphisms in \eqref{eq:iso:VROTF.GRAD} are trivial using \eqref{eq:def:VROTF} and the definitions of the respective co-domains. The other isomorphisms follow from \cite[Corollary 7.3]{Arnold:18}, except in the following situations that can easily be verified by hand: $\ell=0$ in \eqref{eq:iso:DIV}, and $\ell=0$ or $1$ in \eqref{eq:iso:CURL}.
In either case, the connectedness of $F$ or $T$ is crucial for these maps to be isomorphisms.

\subsection{Integration by parts formulas}

We recall a few inspiring integration by parts formulas, starting with those relevant for the design of the discrete gradient and divergence operators.
Given a polyhedron $T\in\Real^3$ and two functions $\bvec{v}_T:T\to\Real^3$ and $q_T:T\to\Real$ smooth enough, we have
\begin{equation}\label{eq:ipp:div.grad.T}
  \int_T\GRAD q_T\cdot\bvec{v}_T
  = -\int_T q_T\DIV\bvec{v}_T + \sum_{F\in\FT}\omega_{TF}\int_F q_T(\bvec{v}_T\cdot\normal_F).
\end{equation}
Similarly, for any polygon $F$ and functions $\bvec{v}_F:F\to\Real^2$ and $q_F:F\to\Real$ smooth enough, we have
\begin{equation}\label{eq:ipp:div.grad.F}
  \int_F\GRAD_F q_F\cdot\bvec{v}_F
  = -\int_F q_F\DIV_F\bvec{v}_F + \sum_{E\in\EF}\omega_{FE}\int_E q_F(\bvec{v}_F\cdot\normal_{FE}),
\end{equation}
while, for any edge $E$ and functions $v_E:E\to\Real$ and $q_E:E\to\Real$ smooth enough,
\begin{equation}\label{eq:ipp:div.grad.E}
  \int_E q_E' v_E
  = -\int_E q_E v_E' + (q_E v_E)(\bvec{x}_{E,2}) - (q_E v_E)(\bvec{x}_{E,1}),
\end{equation}
where the derivatives are taken along the direction $\tangent_E$.

Let us now move to the formulas used in the design of the discrete curl operators.
Given a polyhedron $T\in\Real^3$ and two smooth enough functions $\bvec{v}_T:T\to\Real^3$ and $\bvec{w}_T:T\to\Real^3$, we have that
\begin{equation}\label{eq:ipp:curl.T}
  \begin{aligned}
    \int_T\CURL\bvec{v}_T\cdot\bvec{w}_T
    &= \int_T\bvec{v}_T\cdot\CURL\bvec{w}_T
    + \sum_{F\in\FT}\omega_{TF}\int_F\bvec{v}_T\cdot(\bvec{w}_T\times\normal_F)
    \\
    &= \int_T\bvec{v}_T\cdot\CURL\bvec{w}_T
    + \sum_{F\in\FT}\omega_{TF}\int_F(\normal_F\times(\bvec{v}_T\times\normal_F))\cdot(\bvec{w}_T\times\normal_F),
  \end{aligned}
\end{equation}
the second equality being justified recalling that $\normal_F\times(\bvec{v}_T\times\normal_F)$ is the projection of $\bvec{v}_T$ on the plane spanned by $F$, and noting that $\bvec{w}_T\times\normal_F$ belongs to that plane.
Similarly, for any polygon $F$ and smooth enough functions $\bvec{v}_F:F\to\Real^2$ and $r_F:F\to\Real$,
\begin{equation}\label{eq:ipp:rot.F}
  \int_F\ROT_F\bvec{v}_F~r_F
  = \int_F\bvec{v}_F\cdot\VROT_F r_F
  - \sum_{E\in\EF}\omega_{FE}\int_E(\bvec{v}_F\cdot\tangent_E) r_F.
\end{equation}


\section{An exact two-dimensional sequence}\label{sec:2d.sequence}

In this section we define a discrete counterpart of the following exact two-dimensional sequence on a polygon $F$ (which may be thought of as a mesh face):
\begin{equation}\label{eq:continuous.sequence:2D}
  \begin{tikzcd}
    \Real\arrow{r}{i_F} & H^1(F)\arrow{r}{\GRAD_F} & \Hrot{F}\arrow{r}{\ROT_F} & L^2(F)\arrow{r}{0} & \{0\},
  \end{tikzcd}
\end{equation}
where $i_F$ is the operator that maps a real value to a constant function over $F$ and, with usual notation, $H^1(F)$ denotes the space of functions that are square integrable along with their (tangential) derivatives on $F$, while $\Hrot{F}\coloneq\left\{\bvec{v}\in L^2(F)^2\st\ROT_F\bvec{v}\in L^2(F)\right\}$.
The starting point is, in Section \ref{sec:2d.sequence:consistent.operators}, the design of reconstructions of the two-dimensional gradient and curl operators in full polynomial spaces, which drive the choice of the discrete unknowns. These operators cannot be directly used to form an exact discrete sequence, as we show in Section \ref{sec:2d.sequence:almost.exact}.
Their properties, however, point out to the modifications required to obtain exactness, as detailed in Section \ref{sec:2d.sequence:exact}.
Specifically, it is required to trim the discrete counterpart of the space $\Hrot{F}$ and, correspondingly, replace the full gradient with its projection on this trimmed space.
  The resulting spaces will play the role of the restrictions to polyhedral faces of the three-dimensional spaces of Section \ref{sec:3d.sequence:spaces}; see Remarks \ref{rem:uHgrad.2D.3D} and \ref{rem:uHcurl.2D.3D} below.
Two-dimensional scalar and vector potentials are discussed in Section \ref{sec:2d:potentials}, while the discrete counterparts of $L^2$-products along with their stability and consistency properties make the object of Section \ref{sec:2d.sequence:L2.products}.
From this point on, we fix a polynomial degree $k\ge 0$.

\subsection{Two-dimensional full vector operators reconstructions}\label{sec:2d.sequence:consistent.operators}

We start by defining reconstructions of the vector operators in full polynomial spaces.
As a general convention of notation, we use underlines to denote vectors made of components in different polynomial spaces, and bold fonts for vector-valued polynomials or vectors that have at least one vector-valued polynomial component. %
We use the same convention for operators that map DOFs on vectors of polynomials.
Thus, $\uvec{a}=(\bvec{a},\bvec{b},c)$ denotes the vector whose components are the vector-valued polynomials (or operators with values in vector-valued polynomials) $\bvec{a}$ and $\bvec{b}$ together with the scalar-valued polynomial (or operator with values in scalar-valued polynomials) $c$, while $\underline{a}=(a,b,c)$ is the vector whose components are the scalar-valued polynomials/operators $a$, $b$, and $c$.
We note here the usage of the same letter for both the vector $\uvec{a}$ and one of its components $\bvec{a}$, which will be natural in our definitions of operators.
Also, the full vector operators that will only enter in the sequence through $L^2$-projections or restrictions of their domain will be denoted using a bullet, e.g.~$\accentset{\bullet}{\uvec{a}}$ or $\accentset{\bullet}{\bvec{a}}$, to facilitate their identification.%

\subsubsection{Gradient}

From the polytopal methods literature, it is well known that a consistent gradient can be reconstructed in $\Poly{k}(F)^2$ using polynomials of degree $(k-1)$ inside $F$ and boundary polynomials in $\Poly{k}(\EF)$ (see, e.g., \cite[Section 2.1]{Di-Pietro.Droniou:20}). However, in order for the discrete gradient operator to map on the domain of the discrete curl operator, we also aim here at reconstructing a gradient of degree $k$ on $\partial F$.
For this reason, we will rather consider boundary polynomials in $\Poly[c]{k+1}(\partial F)$. We therefore define
\[
\uGFh\coloneq (\cGF,\cGpF):\Poly{k-1}(F)\times \Poly[c]{k+1}(\partial F)\to \Poly{k}(F)^2\times \Poly{k}(\EF)
\]
such that, for all $\underline{q}_F=(q_F,q_{\partial F})\in \Poly{k-1}(F)\times \Poly[c]{k+1}(\partial F)$,
\begin{equation}\label{eq:GF.k}
  \int_F\cGF\underline{q}_F\cdot\bvec{w}_F
  = -\int_F q_F\DIV_F\bvec{w}_F
  + \sum_{E\in\EF}\omega_{FE}\int_E q_{\partial F}(\bvec{w}_F\cdot\normal_{FE})
  \qquad\forall\bvec{w}_F\in\Poly{k}(F)^2
\end{equation}
and
\begin{equation}\label{eq:GE.k}
  (\cGpF\underline{q}_F)_{|E} = \cGE\underline{q}_F\coloneq (q_{\partial F})_{|E}'\qquad\forall E\in\EF,
\end{equation}
where the derivative on $E$ is taken along the direction $\tangent_E$. Since $\cGpF$ only depends on the boundary values of $\underline{q}_F$, by a slight abuse of notation we also write $\cGpF q_{\partial F}$ instead of $\cGpF \underline{q}_F$ when needed.

We next state two results that will be useful in what follows: the consistency of both components of $\uGFh$, and the surjectivity of $\cGpF$. Let us first introduce the interpolator $\uIgrad[F]: C^0(\overline{F})\to \Poly{k-1}(F)\times \Poly[c]{k+1}(\partial F)$ such that, for all $q\in C^0(\overline{F})$,
\begin{equation}\label{eq:uIgradF}
  \begin{gathered}
    \uIgrad[F] q \coloneq (q_F,q_{\partial F})\in\Poly{k-1}(F)\times \Poly[c]{k+1}(\partial F)\text{ with $q_F=\lproj{k-1}{F} q$,}
    \\
    \text{
      $\lproj{k-1}{E}(q_{\partial F})_{|E}=\lproj{k-1}{E} q_{|E}$ for all $E\in\EF$,
      and $q_{\partial F}(\bvec{x}_V)=q(\bvec{x}_V)$ for all $V\in\VF$.
    }
  \end{gathered}
\end{equation}
The isomorphism \eqref{eq:iso.Pc} shows that the last two relations define $q_{\partial F}$ uniquely.

\begin{proposition}[Polynomial consistency of the reconstructed gradient]
  It holds
  \begin{alignat}{2}
    \label{eq:GF.k.cons}
    \cGF(\uIgrad[F] q) &= \GRAD_F q &\qquad& \forall q\in\Poly{k+1}(F),\\
    \label{eq:GE.k.cons}
    \cGE(\uIgrad[F] q) &= (q_{|E})' &\qquad&\forall q\in\Poly{k+1}(F)\,,\quad\forall E\in\EF.
  \end{alignat}
\end{proposition}

\begin{proof}
  Let $q\in \Poly{k+1}(F)$ and set $\underline{q}_F\coloneq\uIgrad[F]q$. The restriction $q_{|\partial F}$ of $q$ to $\partial F$ obviously satisfies the conditions imposed on $q_{\partial F}$ in \eqref{eq:uIgradF}, and thus $q_{\partial F}=q_{|\partial F}$. This establishes \eqref{eq:GE.k.cons}. The definition \eqref{eq:GF.k} of $\cGF$ then yields, for all $\bvec{w}_F\in\Poly{k}(F)^2$,
  \begin{equation}\label{eq:GF.k.cons.test}
    \begin{aligned}
      \int_F \cGF(\uIgrad[F] q)\cdot\bvec{w}_F
      &= -\int_F (\lproj{k-1}{F} q)\DIV_F\bvec{w}_F
      + \sum_{E\in\EF}\omega_{FE}\int_E q_{\partial F}(\bvec{w}_F\cdot\normal_{FE})
      \\
      &= -\int_F q\DIV_F\bvec{w}_F
      + \sum_{E\in\EF} \omega_{FE}\int_E q_{|\partial F}(\bvec{w}_F\cdot\normal_{FE})=\int_F \GRAD_F q \cdot\bvec{w}_F,
    \end{aligned}
  \end{equation}
  where the removal of $\lproj{k-1}{F}$ in the second line is justified by its definition \eqref{eq:lproj} along with the fact that $\DIV_F\bvec{w}_F\in\Poly{k-1}(F)$, and the conclusion follows from the integration by parts formula \eqref{eq:ipp:div.grad.F}.
  Since both $\cGF(\uIgrad[F] q)$ and $\GRAD_F q$ belong to $\Poly{k}(F)^2$, \eqref{eq:GF.k.cons.test} proves \eqref{eq:GF.k.cons}.
\end{proof}

\begin{proposition}[Surjectivity of $\cGpF$]\label{prop:surj.GbdryF}
  For all $r_{\partial F}\in \Poly{k}(\EF)$ such that $\sum_{E\in\EF}\omega_{FE}\int_E r_{\partial F}=0$, there
  exists $q_{\partial F}\in\Poly[c]{k+1}(\partial F)$ such that $\cGpF q_{\partial F}=r_{\partial F}$.
\end{proposition}

\begin{remark}[Bijectivity of $\cGpF$]
  It is not difficult to check that the condition $\sum_{E\in\EF}\omega_{FE}\int_E r_{\partial F}=0$ is also necessary for $r_{\partial F}$ to be in the image of $\cGpF$, which is therefore an isomorphism between $\Poly[c]{k+1}(\partial F)$ and $\left\{r_{\partial F}\in\Poly{k}(\EF)\st\sum_{E\in\EF}\omega_{FE}\int_Er_{\partial F}=0\right\}$.
\end{remark}

\begin{proof}[Proof of Proposition \ref{prop:surj.GbdryF}]
  Define the function $\tilde{r}_{\partial F}:\partial F\to\Real$ by setting $(\tilde{r}_{\partial F})_{|E}\coloneq\omega_{FE}(r_{\partial F})_{|E}$ for all $E\in\EF$.
  Then,
  \begin{equation}\label{eq:rtilde.zero}
    \int_{\partial F}\tilde{r}_{\partial F}=0.
  \end{equation}
  Fix an arbitrary $V\in\VF$ and, for a given $\bvec{x}\in\partial F$, let $\Gamma_{\bvec{x}_V\to\bvec{x}}$ be the path in $\partial F$ that goes from $\bvec{x}_V$ to $\bvec{x}$ in a clockwise direction.
  The connectedness of $\partial F$ and the Lipschitz boundary property of $F$ (which implies that $\partial F$ is a simple curve) ensure that this path covers all of $\partial F$, and that it does not loop on itself before coming back to $\bvec{x}_V$.
  Define then $q_{\partial F}(\bvec{x})$ as the integral of $\tilde{r}_{\partial F}$ along $\Gamma_{\bvec{x}_V\to\bvec{x}}$. The condition \eqref{eq:rtilde.zero} ensures the continuity of $q_{\partial F}$ at $V$ after a complete loop around $\partial F$. By construction, the derivative of $q_{\partial F}$ in the clockwise direction along $\partial F$ is
  equal to $\tilde{r}_{\partial F}$. This means that, on any $E\in\EF$ with orientation $\tangent_E$, we have $\omega_{FE}(q_{\partial F})_{|E}'=(\tilde{r}_{\partial F})_{|E}=\omega_{FE}(r_{\partial F})_{|E}$,
  which precisely establishes $\cGpF q_{\partial F}=r_{\partial F}$. 
\end{proof}

\subsubsection{Curl}

The full two-dimensional scalar curl reconstruction operator is $\hcCF:\Poly{k}(F)^2\times \Poly{k}(\EF)\to\Poly{k}(F)$ such that, for all $\uvec{v}_F=(\bvec{v}_F,v_{\partial F})\in \Poly{k}(F)^2\times \Poly{k}(\EF)$,
\begin{equation}\label{eq:chCF}
  \int_F\hcCF\uvec{v}_F~r_F = \int_F\bvec{v}_F\cdot\VROT_F r_F - \sum_{E\in\EF}\omega_{FE}\int_E v_{\partial F} r_F\qquad\forall r_F\in\Poly{k}(F).
\end{equation}
Define the interpolator $\uhIrot[F]:H^1(F)^2\to \Poly{k}(F)^2\times \Poly{k}(\EF)$ such that, for all $\bvec{v}\in H^1(F)^2$,
\begin{equation}\label{eq:uIhrot}
  \uhIrot\bvec{v}\coloneq\big(
  \vlproj{k}{F}\bvec{v},
  (\lproj{k}{E}(\bvec{v}\cdot\tangent_E)\big)_{E\in\EF}
  \big).
\end{equation}

\begin{proposition}[Commutation property for $\hcCF$]\label{prop:CFh.commutation}
  The following commutation property holds:
  \begin{equation}\label{eq:CFh.commutation}
    \hcCF(\uhIrot\bvec{v}) = \lproj{k}{F}(\ROT_F\bvec{v})\qquad\forall\bvec{v}\in H^1(F)^2.
  \end{equation}
\end{proposition}

\begin{proof}
  Writing \eqref{eq:chCF} for $\uvec{v}_F=\uhIrot\bvec{v}$ we have, for all $r_F\in\Poly{k}(F)$,
  \begin{equation}\label{eq:comm.CF.1}
    \begin{aligned}
      \int_F\hcCF(\uhIrot\bvec{v})~r_F
      &= \int_F\vlproj{k}{F}\bvec{v}\cdot\VROT_F r_F
      - \sum_{E\in\EF}\omega_{FE}\int_E \lproj{k}{E}(\bvec{v}\cdot\tangent_E)~r_F
      \\
      &= \int_F\bvec{v}\cdot\VROT_F r_F
      - \sum_{E\in\EF}\omega_{FE}\int_E (\bvec{v}\cdot\tangent_E)~r_F
      = \int_F\ROT_F\bvec{v}~r_F,	
    \end{aligned}
  \end{equation}
  where, to remove the $L^2$-orthogonal projectors in the second line, we have used their definition \eqref{eq:lproj} after observing that $\VROT_F r_F\in\Poly{k-1}(F)$ and $(r_F)_{|E}\in\Poly{k}(E)$ for all $E\in\EF$,
  and we have invoked the integration by parts formula \eqref{eq:ipp:rot.F} to conclude.
  By definition of $\lproj{k}{F}$, \eqref{eq:comm.CF.1} implies \eqref{eq:CFh.commutation}.
\end{proof}

\begin{remark}[Internal unknown]\label{rem:CFh:face.unknown}
  An inspection of the above proof reveals that the commutation property \eqref{eq:CFh.commutation} holds also if we take $\Roly{k-1}(F)\times\Poly{k}(\EF)$ instead of $\Poly{k}(F)^2\times \Poly{k}(\EF)$ as a domain for the discrete curl operator, and we correspondingly replace $\vlproj{k}{F}$ with $\Rproj{k-1}{F}$ in the definition \eqref{eq:uIhrot} of $\uhIrot$.
\end{remark}

\subsection{An almost-exact two-dimensional sequence}\label{sec:2d.sequence:almost.exact}

The two-dimensional full gradient and curl reconstructions define the following sequence:
\begin{equation}\label{eq:almost.exact}
  \begin{tikzcd}
    \Real\arrow{r}{\uIgrad[F]} & \Poly{k-1}(F)\times \Poly[c]{k+1}(\partial F)\arrow{r}{\uGFh} & \Poly{k}(F)^2\times \Poly{k}(\EF)\arrow{r}{\hcCF} & \Poly{k}(F)\arrow{r}{0} & \{0\}.
  \end{tikzcd}
\end{equation}
This sequence satisfies some exactness properties, but is not completely exact.
Analysing these properties, as done in the following proposition, will guide us to define an exact two-dimensional sequence of spaces and operators.

\begin{proposition}[Properties of the sequence \eqref{eq:almost.exact}]\label{prop:almost.exact}
  It holds:
  \begin{equation}
    \uIgrad[F]\Real = \Ker\uGFh \label{eq:exact.uGFh.ker},
  \end{equation}
  \begin{equation}
    \Image\uGFh\subset\Ker\hcCF \label{eq:exact.CF.ker.1},
  \end{equation}
  \begin{equation}\label{eq:exact.CF.ker.2}
    \begin{gathered}
      \text{For all $\uvec{v}_F=(\bvec{v}_F,v_{\partial F})\in\Ker\hcCF$, there exists $\underline{q}_F=(q_F,q_{\partial F})\in \Poly{k-1}(F)\times \Poly[c]{k+1}(\partial F)$}
      \\
      \text{such that
        $v_{\partial F}=\cGpF q_{\partial F}$,
        $\Rproj{k-1}{F}\bvec{v}_F=\Rproj{k-1}{F}(\cGF\underline{q}_F)$,
        and $\ROproj{k}{F}\bvec{v}_F = \ROproj{k}{F}(\cGF\underline{q}_F)$,
      }
    \end{gathered}
  \end{equation}
  and
  \begin{equation}
    \Image\hcCF = \Poly{k}(F).
    \label{eq:exact.CF.Im}
  \end{equation}
\end{proposition}

\begin{proof}
  \underline{1. \emph{Proof of \eqref{eq:exact.uGFh.ker}.}}
  If $C\in\Real$ and $\underline{q}_F=\uIgrad[F]C$, then the consistency properties \eqref{eq:GF.k.cons} and \eqref{eq:GE.k.cons} give, respectively, $\cGF\underline{q}_F= 0$ and $\cGpF\underline{q}_F=0$.
  These two relations establish the inclusion $\subset$ in \eqref{eq:exact.uGFh.ker}.

  To prove the converse inclusion, we first notice that, if $\underline{q}_F=(q_F,q_{\partial F})\in \Poly{k-1}(F)\times \Poly[c]{k+1}(F)$ is such that $\uGFh\underline{q}_F=\uvec{0}$, then $\cGpF q_{\partial F}=0$ and hence $q_{\partial F}=C$ for some $C\in\Real$ (since $q_{\partial F}$ is continuous and its derivative vanishes on each $E\in\EF$, and $\partial F$ is connected). Plugging this result into the definition \eqref{eq:GF.k} of $\cGF$ and using the integration by parts formula \eqref{eq:ipp:div.grad.F} with $C$ instead of $q_F$, we infer, for all $\bvec{w}_F\in\Poly{k}(F)^2$,
  \begin{equation}\label{eq:for.exact.proj}
    \int_F\cGF\underline{q}_F\cdot\bvec{w}_F=-\int_F q_F\DIV_F\bvec{w}_F
    + \sum_{E\in\EF}\omega_{FE}\int_E C(\bvec{w}_F\cdot\normal_{FE})
    =\int_F (C-q_F)\DIV_F\bvec{w}_F.
  \end{equation}
  Using $\uGFh\underline{q}_F=\uvec{0}$, we see that the left-hand side vanishes and, since $\DIV_F:\Poly{k}(F)^2\to \Poly{k-1}(F)$ is surjective (consequence of \eqref{eq:iso:DIV}) and $q_F\in\Poly{k-1}(F)$,
  this yields $\lproj{k-1}{F}C=\lproj{k-1}{F}q_F=q_F$. Together with $q_{\partial F}=C$, this establishes that $\uIgrad[F]C=\underline{q}_F$, which concludes the proof of the inclusion $\supset$ in \eqref{eq:exact.uGFh.ker}.
  \medskip\\
  \underline{2. \emph{Proof of \eqref{eq:exact.CF.ker.1}.}}
  Let $\underline{q}_F\in \Poly{k-1}(F)\times \Poly[c]{k+1}(\partial F)$ and write, using the definition \eqref{eq:chCF} of $\hcCF$, for all $r_F\in\Poly{k}(F)$,
  \begin{align*}
    \int_F\hcCF(\uGFh\underline{q}_F) r_F={}& \int_F \cGF\underline{q}_F\cdot\VROT_F r_F - 
    \sum_{E\in\EF}\omega_{FE}\int_E \cGpF\underline{q}_F r_F\\
    ={}& \sum_{E\in\EF}\omega_{FE}\int_E\left[
      q_{\partial F}\VROT_F r_F \cdot\normal_{FE}-(q_{\partial F})_{|E}' r_F
      \right],
  \end{align*}
  where the second line follows using the definitions \eqref{eq:GF.k} of $\cGF$ with $\bvec{w}_F=\VROT_F r_F$ (additionally noticing that $\DIV_F(\VROT_F r_F)=0$) and \eqref{eq:GE.k} of $\cGpF$. We then use the integration by parts formula \eqref{eq:ipp:div.grad.E} on each $E\in\EF$ to obtain
  \begin{equation}\label{eq:cancel.qV}
    \begin{aligned}
      \int_F\hcCF(\uGFh\underline{q}_F) r_F={}& \sum_{E\in\EF}\omega_{FE}\int_E q_{\partial F}\cancel{\left(\VROT_F r_F \cdot\normal_{FE}+(r_F)_{|E}'\right)}
      \\
      &- \sum_{E\in\EF}\omega_{FE}\left(
      q_{\partial F}(\bvec{x}_{E,2})r_F(\bvec{x}_{E,2}) - q_{\partial F}(\bvec{x}_{E,1}) r_F(\bvec{x}_{E,1})
      \right)=0,
    \end{aligned}
  \end{equation}
  where the cancellation comes from
\begin{equation}\label{eq:vrotF.dE}
\VROT_F r_F\cdot\normal_{FE}=\GRAD_F r_F\cdot(\rotation{-\nicefrac\pi2}^t\normal_{FE})=\GRAD_F r_F\cdot(\rotation{\nicefrac\pi2}\normal_{FE})=-\GRAD_F r_F\cdot\tangent_E=-(r_F)_{|E}'
\end{equation}
(since $(\tangent_E,\normal_{FE})$ is right-handed in $F$), and we have concluded using the fact that, for all $V\in\VT$, the term $q_{\partial F}(\bvec{x}_V) r_F(\bvec{x}_V)$ appears exactly twice in the last sum, with opposite signs. This proves \eqref{eq:exact.CF.ker.1}.
  \medskip\\
  \underline{3. \emph{Proof of \eqref{eq:exact.CF.ker.2}.}}
  Suppose that $\uvec{v}_F\in \Poly{k}(F)^2\times\Poly{k}(\EF)$ is such that $\hcCF\uvec{v}_F=0$. Then, \eqref{eq:chCF} with $r_F=1$ shows that $\sum_{E\in\EF}\omega_{FE}\int_E v_{\partial F}=0$ and thus, by Proposition \ref{prop:surj.GbdryF}, there exists $q_{\partial F}\in\Poly[c]{k+1}(\partial F)$ such that $\cGpF q_{\partial F}=v_{\partial F}$. This establishes the first conclusion in \eqref{eq:exact.CF.ker.2}.

  Using $\hcCF\uvec{v}_F=0$ and again the definition \eqref{eq:chCF} of $\hcCF$, for all $r_F\in\Poly{k}(F)$ we then have
  \begin{equation}\label{eq:exact.CF.middle}
    \begin{aligned}
      \int_F \bvec{v}_F\cdot\VROT_Fr_F
      &= \sum_{E\in\EF}\omega_{FE}\int_E \cGpF q_{\partial F}~r_F\\
      &= -\sum_{E\in\EF}\omega_{FE}\int_E q_{\partial F} (r_F)_{|E}'\\
      &= \sum_{E\in\EF}\omega_{FE}\int_E q_{\partial F}(\VROT_F r_F\cdot\normal_{FE})\\
      &= \int_F \cGF(q_F,q_{\partial F}) \cdot\VROT_F r_F,
    \end{aligned}
  \end{equation}
  where the second line follows integrating by parts on each edge and cancelling out the vertex values in a similar way as in \eqref{eq:cancel.qV},
  the third line from \eqref{eq:vrotF.dE},
  and the conclusion is obtained applying the definition \eqref{eq:GF.k} of $\cGF$ to $\bvec{w}_F=\VROT_F r_F$ (which satisfies $\DIV_F \bvec{w}_F=0$) and $\underline{q}_F=(q_F,q_{\partial F})$ for an arbitrary $q_F\in\Poly{k-1}(F)$. Since \eqref{eq:exact.CF.middle} is valid for any $r_F\in\Poly{k}(F)$, this proves the second conclusion in \eqref{eq:exact.CF.ker.2}.

  To conclude the proof of \eqref{eq:exact.CF.ker.2}, we need to identify a specific $q_F\in\Poly{k-1}(F)$ such that 
  $\ROproj{k}{F}\bvec{v}_F = \ROproj{k}{F}(\cGF\underline{q}_F)$, that is to say, for any $\bvec{w}_F\in\Roly{k}(F)^\perp$, 
  \begin{align*}
    \int_F \bvec{v}_F \cdot \bvec{w}_F
    ={} \int_F \cGF\underline{q}_F\cdot\bvec{w}_F
    ={} -\int_F q_F \DIV_F\bvec{w}_F + \sum_{E\in\EF}\omega_{FE}\int_E q_{\partial F}(\bvec{w}_F\cdot\normal_{FE}),
  \end{align*}
  where we have used the definition \eqref{eq:GF.k} of $\cGF$ in the second passage.
  Since $q_{\partial F}$ is already given, we simply have to take $q_F\in\Poly{k-1}(F)$ such that:
  \[
  \int_F q_F \DIV_F\bvec{w}_F = -\int_F \bvec{v}_F \cdot \bvec{w}_F + \sum_{E\in\EF}\omega_{FE}\int_E q_{\partial F}(\bvec{w}_F\cdot\normal_{FE})\qquad\forall \bvec{w}_F\in\Roly{k}(F)^\perp.
  \]
  By \eqref{eq:iso:DIV}, this relation defines $q_F$ uniquely.
  \medskip\\
  \underline{4. \emph{Proof of \eqref{eq:exact.CF.Im}.}}
  We only have to prove $\Poly{k}(F)\subset \Image\hcCF$. Let $q_F\in\Poly{k}(F)$. Since $\ROT_F:\Poly{k+1}(F)^2\to \Poly{k}(F)$ is surjective (this is a consequence of its definition \eqref{eq:def:ROTF} along with the surjectivity of $\DIV_F:\Poly{k+1}(F)^2\to\Poly{k}(F)$, which follows from \eqref{eq:iso:DIV}), there
  is $\bvec{v}\in\Poly{k+1}(F)^2$ such that $\ROT_F\bvec{v}=q_F$. Hence, using the polynomial consistency of $\lproj{k}{F}$ followed by the commutation property \eqref{eq:CFh.commutation}, we have $q_F=\ROT_F\bvec{v}=\lproj{k}{F}(\ROT_F\bvec{v})=\hcCF(\uhIrot\bvec{v})\in\Image\hcCF$, which is the desired result.
\end{proof}

\subsection{An exact two-dimensional sequence}\label{sec:2d.sequence:exact}

Proposition \ref{prop:almost.exact} shows that the defect of exactness of the sequence \eqref{eq:almost.exact} lies in the domain of $\hcCF$/co-domain of $\uGFh$, which is too large.
Specifically, the space $\Poly{k}(F)^2\times \Poly{k}(\EF)$ in this sequence must be restricted to its subspace $\left(\Roly{k-1}(F)\oplus \Roly{k}(F)^\perp\right)\times \Poly{k}(\EF)$, which still contains sufficient information to reconstruct a discrete curl satisfying the commutation property \eqref{eq:CFh.commutation} (cf. Remark \ref{rem:CFh:face.unknown}).
Obviously, this restriction requires to project $\uGFh$ onto this space in order for the sequence to be well-defined.

The domain of the reconstructed gradient does not change, so we take as discrete counterpart of the space $H^1(F)$ in the sequence \eqref{eq:continuous.sequence:2D} the space
\[
\uHgrad[F]\coloneq \Poly{k-1}(F)\times \Poly[c]{k+1}(\partial F)
\]
and, as before, a generic vector $\underline{q}_F\in \uHgrad[F]$ is denoted by $(q_F,q_{\partial F})$ with $q_F\in\Poly{k-1}(F)$ and $q_{\partial F}\in\Poly[c]{k+1}(\partial F)$.
The interpolator on $\uHgrad[F]$ does not change either:
\[
\begin{gathered}
  \uIgrad[F] q \coloneq (q_F,q_{\partial F})\in\Poly{k-1}(F)\times \Poly[c]{k+1}(\partial F)\text{ with $q_F=\lproj{k-1}{F} q$,}
  \\
  \text{
    $\lproj{k-1}{E}(q_{\partial F})_{|E}=\lproj{k-1}{E} q_{|E}$ for all $E\in\EF$,
    and $q_{\partial F}(\bvec{x}_V)=q(\bvec{x}_V)$ for all $V\in\VF$.
  }
\end{gathered}
\]
The domain of the reconstructed curl, which plays the role of the space $\Hrot{F}$ at the discrete level, is now
\[
\uHrot\coloneq \left(\Roly{k-1}(F)\oplus \Roly{k}(F)^\perp\right)\times \Poly{k}(\EF),
\]
and a generic vector $\uvec{v}_F\in\uHrot$ is decomposed into $(\bvec{v}_F=\bvec{v}_{\cvec{R},F}+\bvec{v}_{\cvec{R},F}^\perp,v_{\partial F})$ with $\bvec{v}_{\cvec{R},F}\in \Roly{k-1}(F)$, $\bvec{v}_{\cvec{R},F}^\perp\in \Roly{k}(F)^\perp$, and $v_{\partial F}\in\Poly{k}(\EF)$.

The discrete gradient operator $\uGF:\uHgrad[F]\to \uHrot$ is defined by projecting the components of $\uGFh$ onto the corresponding components of $\uHrot$:
For all $\underline{q}_F\in\uHgrad[F]$,
\begin{equation}\label{eq:def.GF.GOF}
  \begin{gathered}
    \uGF\underline{q}_F\coloneq(\GRF\underline{q}_F+\GROF\underline{q}_F,\cGpF q_{\partial F})
    \\
    \text{with $\GRF\coloneq\Rproj{k-1}{F}\cGF$ and $\GROF\coloneq\ROproj{k}{F}\cGF$.}
  \end{gathered}
\end{equation}
It can easily be checked from \eqref{eq:GF.k} that the following two relations characterise $\GRF$ and $\GROF$: For all $\underline{q}_F\in\uHgrad[F]$,
\begin{subequations}\label{eq:uGF}
  \begin{alignat}{2}\label{eq:GF}      
    \int_F\GRF\underline{q}_F\cdot\bvec{w}_F
    &= \sum_{E\in\EF}\omega_{FE}\int_E q_{\partial F} (\bvec{w}_F\cdot\normal_{FE})
    &\quad& \forall\bvec{w}_F\in\Roly{k-1}(F),
    \\ \label{eq:GOF}
    \int_F\GROF\underline{q}_F\cdot\bvec{w}_F
    &=
    -\int_F q_F\DIV_F\bvec{w}_F
    +\sum_{E\in\EF}\omega_{FE}\int_E q_{\partial F} (\bvec{w}_F\cdot\normal_{FE})
    &\;& \forall\bvec{w}_F\in\Roly{k}(F)^\perp.
  \end{alignat}
\end{subequations}
The discrete curl operator $\cCF:\uHrot\to\Poly{k}(F)$ is given by the restriction of $\hcCF$ to $\uHrot$, that is:
For all $\uvec{v}_F=(\bvec{v}_{\cvec{R},F}+\bvec{v}_{\cvec{R},F}^\perp, v_{\partial F})\in\uHrot$,
\begin{equation}\label{eq:cCF}
  \int_F\cCF\uvec{v}_F~r_F
  = \int_F\bvec{v}_{\cvec{R},F}\cdot\VROT_F r_F
  - \sum_{E\in\EF}\omega_{FE}\int_E v_{\partial F}r_F\qquad
  \forall r_F\in\Poly{k}(F).
\end{equation}
Notice that, in the integral over $F$, we have removed the component $\bvec{v}_{\cvec{R},F}^\perp$ of $\bvec{v}_F$ accounting for the fact that it is $L^2$-orthogonal to $\VROT_F r_F\in\Roly{k-1}(F)\subset\Roly{k}(F)$.

Letting $\uIrot:H^1(F)^2\to\uHrot$ be the interpolator obtained projecting $\uhIrot$, that is, for all $\bvec{v}\in H^1(F)^2$,
\begin{equation}\label{eq:uIrot}
  \uIrot\bvec{v}\coloneq\big(
  \Rproj{k-1}{F}\bvec{v}+
  \ROproj{k}{F}\bvec{v},
  (\lproj{k}{E}(\bvec{v}\cdot\tangent_E)\big)_{E\in\EF}
  \big),
\end{equation}
we have, following Remark \ref{rem:CFh:face.unknown}, the commutation property
\begin{equation}\label{eq:CF.commutation}
  \cCF(\uIrot\bvec{v}) = \lproj{k}{F}(\ROT_F\bvec{v})\qquad\forall\bvec{v}\in H^1(F)^2.
\end{equation}

\begin{theorem}[Exact two-dimensional sequence]\label{th:exact.2D}
  The following sequence is exact:
  \begin{equation}\label{eq:sequence.2D}
    \begin{tikzcd}
      \Real\arrow{r}{\uIgrad[F]} & \uHgrad[F]\arrow{r}{\uGF} & \uHrot\arrow{r}{\cCF} & \Poly{k}(F)\arrow{r}{0} & \{0\}.
    \end{tikzcd}
  \end{equation}
\end{theorem}

\begin{remark}[Comparison with Finite Elements]\label{rem:comparison.2D}
  When $F$ is a triangle or a rectangle, the sequence \eqref{eq:sequence.2D} can be compared with classical FE sequences: for the triangular case, \mbox{($\Poly{k+1}(F)$, $\rotation{\pi/2}\RT{k}(F)$, $\Poly{k}(F)$)}, with $\rotation{\pi/2}\RT{k}(F)$ denoting the rotated two-dimensional Raviart--Thomas space \cite{Raviart.Thomas:77}; for the rectancular case, \mbox{($\Qoly{k+1}(F)$, $\bQoly{k,-}(F)$, $\Qoly{k}(F)$)}, where $\Qoly{\ell}(F)$ denotes the space of polynomials of total degree $\le\ell$ in each variable and $\bQoly{k,-}$ is the space defined, e.g. in \cite{Arnold:18}.
  The number of discrete unknowns in each case is reported in Table \ref{tab:comparison.2D}.
  On triangles, the number of discrete unknowns attached to the lowest-dimensional geometric support is the same as for Finite Elements for each space in the sequence, but our sequence has more internal unknowns.
  This phenomenon is known from the Virtual Element literature and can be countered using serendipity spaces; see, e.g., \cite{Beirao-da-Veiga.Brezzi.ea:18*2}.
  Notice, however, that it is not clear whether resorting to serendipity versions of the proposed spaces would have an impact on the exactness property: this point is left for future investigation.
  On rectangles, on the other hand, our sequence has fewer unknowns than the corresponding Finite Elements sequence.
  
  Two points are worth mentioning to close this remark.
  First, when writing a scheme, internal unknowns can usually be locally eliminated by static condensation; see, e.g., \cite[Section B.3.2]{Di-Pietro.Droniou:20}.
  This strategy essentially mitigates the impact of the number of internal unknowns in practical computations.
  Second, the proposed construction can in principle be combined with standard Finite Elements on meshes that contain both standard and polygonal/polyhedral elements.  
\end{remark}

\begin{table}\centering
  \begin{small}
    \begin{tabular}{c|c|c|c|cc|cc}
      \toprule
      & \multirow{2}{*}{$k$} & \multirow{2}{*}{$V\in\VT$} & \multirow{2}{*}{$E\in\ET$} & \multicolumn{2}{c|}{$F\in\FT$} & \multicolumn{2}{c}{Total}
      \\
      & & & & Tria & Rect & Tria & Rect
      \\\midrule
      \multirow{4}{*}{$\uHgrad[F]$} 
      & 0 & \cellcolor{black!10}{1} & \cellcolor{black!10}{0} & 0 (0) & 0 (0) & 3 (3)   & 4 (4)
      \\
      & 1 & \cellcolor{black!10}{1} & \cellcolor{black!10}{1} & 1 (0) & 1 (1) & 7 (6)   & 9 (9)
      \\
      & 2 & \cellcolor{black!10}{1} & \cellcolor{black!10}{2} & 3 (1) & 3 (4) & 12 (10) & 15 (16)
      \\
      & 3 & \cellcolor{black!10}{1} & \cellcolor{black!10}{3} & 6 (3) & 6 (9) & 18 (15) & 22 (25)    
      \\\midrule
      \multirow{4}{*}{$\uHrot$} 
      & 0 & & \cellcolor{black!10}{1} & 0 (0)   & 0 (0)   & 3 (3)   & 4 (4)
      \\
      & 1 & & \cellcolor{black!10}{2} & 3 (2)   & 3 (4)   & 9 (8)   & 11 (12)
      \\
      & 2 & & \cellcolor{black!10}{3} & 8 (6)   & 8 (12) & 17 (15) & 20 (24)
      \\
      & 3 & & \cellcolor{black!10}{4} & 15 (12) & 15 (24) & 27 (24) & 31 (40)
      \\\midrule
      \multirow{4}{*}{$\Poly{k}(F)$} 
      & 0 & & & \cellcolor{black!10}{1} & 1 (1) & \cellcolor{black!10}{1} & 1 (1)
      \\
      & 1 & & & \cellcolor{black!10}{3} & 3 (4) & \cellcolor{black!10}{3} & 3 (4)
      \\
      & 2 & & & \cellcolor{black!10}{6} & 6 (9) & \cellcolor{black!10}{6} & 6 (9)
      \\
      & 3 & & & \cellcolor{black!10}{10} & 10 (16) & \cellcolor{black!10}{10} & 10 (16)
      \\\bottomrule
    \end{tabular}
  \end{small}
  \caption{Number of discrete unknowns attached to each geometric entity for the two-dimensional sequence \eqref{eq:sequence.2D} on a triangle or rectangle $F$ for $k\in\{0,\ldots,3\}$.
    For comparison, we also report in parentheses the number of DOFs of the spaces in the corresponding FE sequences when the latter are different.
    Gray cells highlight the cases for which, irrespectively of the degree $k$, the number of DOFs are the same for the two-dimensional sequence \eqref{eq:sequence.2D} and for the FE sequences.
    \label{tab:comparison.2D}}  
\end{table}

\begin{proof}[Proof of Theorem \ref{th:exact.2D}]
  We have to prove that
  \begin{align}\label{eq:Ker.uGF}
    \uIgrad[F]\Real &= \Ker\uGF,
    \\\label{eq:Ker.CF}
    \Image\uGF &= \Ker\cCF,
    \\\label{eq:Image.CF}
    \Image\cCF &= \Poly{k}(F).
  \end{align}
  \\
  \underline{1. \emph{Proof of \eqref{eq:Ker.uGF}.}} 
  Since $\uGF$ is a projection of $\uGFh$, we have $\Ker\uGFh\subset\Ker\uGF$ and, thus, \eqref{eq:exact.uGFh.ker} gives
  \[
  \boxed{\uIgrad[F]\Real\subset \Ker\uGF.}
  \]

  Assume now that $\underline{q}_F\in\uHgrad[F]$ is such that $\uGF\underline{q}_F=\uvec{0}$. As in the proof of \eqref{eq:exact.uGFh.ker}, since the boundary components of $\uGF$ and $\uGFh$ are the same, this shows that $q_{\partial F}=C$ for some constant $C\in\Real$.
  Moreover, equation \eqref{eq:for.exact.proj} still holds, and can be applied to a generic $\bvec{w}_F\in\Roly{k}(F)^\perp$ to infer
  \[
  \int_F(C-q_F)\DIV_F \bvec{w}_F=\int_F \cGF\underline{q}_F\cdot\bvec{w}_F=\int_F \GROF \underline{q}_F\cdot\bvec{w}_F = 0,
  \]
  where the second equality follows from the definition of $\ROproj{k}{F}$ recalling that $\GROF=\ROproj{k}{F}\cGF$ (see \eqref{eq:def.GF.GOF}.
  Together with the surjectivity of $\DIV_F:\Roly{k}(F)^\perp\to\Poly{k-1}(F)$, this ensures as before that $q_F=\lproj{k-1}{F}C$, which concludes the proof that $\underline{q}_F=\uIgrad[F]C$.
  Hence, we have
  \[
  \boxed{\Ker\uGF\subset\uIgrad[F]\Real,}
  \]
  concluding the proof of \eqref{eq:Ker.uGF}.
  \medskip\\
  \underline{2. \emph{Proof of \eqref{eq:Ker.CF}.}}
  Let $\underline{q}_F\in \uHgrad[F]$ and set $\uvec{v}_F\coloneq\uGF\underline{q}_F$, that is, by definition \eqref{eq:def.GF.GOF} of $\uGF$,
  $\uvec{v}_F=(\bvec{v}_F=\Rproj{k-1}{F}\bvec{w}_F+\ROproj{k}{F}\bvec{w}_F,w_{\partial F})$ with $\uvec{w}_F=\uGFh\underline{q}_F$. By \eqref{eq:exact.CF.ker.1}, we have $\hcCF\uvec{w}_F=0$.
  Using the definitions \eqref{eq:chCF} and \eqref{eq:cCF} of $\hcCF\uvec{w}_F$ and $\cCF\uvec{v}_F$ and the relation
  \[
  \int_F \bvec{w}_F\cdot\VROT_F r_F
  = \int_F \Rproj{k-1}{F}\bvec{w}_F\cdot\VROT_Fr_F
  = \int_F \bvec{v}_{\cvec{R},F}\cdot\VROT_Fr_F\qquad\forall r_F\in\Poly{k}(F),
  \]
  which follows from the definitions of $\Rproj{k-1}{F}$ and $\bvec{v}_{\cvec{R},F}$,
  we see that $0=\hcCF\uvec{w}_F=\cCF\uvec{v}_F$. Hence, $\cCF(\uGF\underline{q}_F)=0$ and
  \begin{equation}\label{eq:Image.uGF.subset.Ker.CF}
    \boxed{\Image\uGF\subset\Ker\cCF.}
  \end{equation}

  Let now $\uvec{v}_F=(\bvec{v}_F=\bvec{v}_{\cvec{R},F}+\bvec{v}_{\cvec{R},F}^\perp,v_{\partial F})\in\Ker\cCF=\uHrot\cap\Ker\hcCF$.
  By \eqref{eq:exact.CF.ker.2}, there is $\underline{q}_F\in\Poly{k-1}(F)\times\Poly[c]{k+1}(\partial F)=\uHgrad[F]$ such that $v_{\partial F}=\cGpF q_{\partial F}$, $\Rproj{k-1}{F}\bvec{v}_F=\Rproj{k-1}{F}(\cGF\underline{q}_F)=\GRF\underline{q}_F$ and $\ROproj{k}{F}\bvec{v}_F=\ROproj{k}{F}(\cGF\underline{q}_F)=\GROF\underline{q}_F$. Since $\Roly{k-1}(F)\subset \Roly{k}(F)$ is orthogonal to $\Roly{k}(F)^\perp$, we have $\Rproj{k-1}{F}\bvec{v}_F=\bvec{v}_{\cvec{R},F}$ and $\ROproj{k}{F}\bvec{v}_F=\bvec{v}_{\cvec{R},F}^\perp$, which proves that $\uvec{v}_F=\uGF \underline{q}_F$. Hence,
  \[
  \boxed{\Ker\cCF\subset \Image\uGF}
  \]
  and the second exactness property \eqref{eq:Ker.CF} is proved.
  \medskip\\
  \underline{3. \emph{Proof of \eqref{eq:Image.CF}.}}
  Consequence of the commutation property \eqref{eq:CF.commutation} proceeding as in the proof of \eqref{eq:exact.CF.Im}.
\end{proof}

\subsection{Two-dimensional potentials}\label{sec:2d:potentials}

We next exhibit reconstructions of the potential for each space in the sequence.

\subsubsection{Scalar potential (scalar trace)}\label{sec:trF}

We start with a scalar potential reconstruction $\trF:\uHgrad[F]\to\Poly{k+1}(F)$ which, when $F$ represents a face of a polyhedron $T$, plays the role of a reconstruction of the scalar trace on $F$. The required properties on $\trF$ are the following:
\begin{subequations}\label{eq:trF:prop}
  \begin{alignat}{2}
    \label{eq:trF:polynomial.consistency}
    &\trF(\uIgrad[F] q) = q&\qquad&\forall q\in\Poly{k+1}(F),\\
    \label{eq:trF:projection}
    &\lproj{k-1}{F}(\trF\underline{q}_F)= q_F&\qquad&\forall \underline{q}_F=(q_F,q_{\partial F})\in\uHgrad[F].
  \end{alignat}
\end{subequations}
The first property expresses the polynomial consistency of the reconstruction, while the second enforces its projection on $\Poly{k-1}(F)$ and shows that $\trF$ has to be a higher-order correction of the projection $\uHgrad[F]\ni\underline{q}_F=(q_F,q_{\partial F})\mapsto q_F\in\Poly{k-1}(F)$.
It is easily checked that, if $\trFtilde:\uHgrad[F]\to\Poly{k+1}(F)$ is a reconstruction that satisfies the consistency property \eqref{eq:trF:polynomial.consistency}, then a reconstruction $\trF$ satisfying \eqref{eq:trF:polynomial.consistency}--\eqref{eq:trF:projection} is obtained setting
\[
\trF\underline{q}_F = q_F + \trFtilde\underline{q}_F-\lproj{k-1}{F}(\trFtilde\underline{q}_F)\qquad\forall\underline{q}_F\in\uHgrad[F].
\]

\begin{remark}[A consistent potential reconstruction]
  There are several ways to devise a reconstruction $\trFtilde:\uHgrad[F]\to\Poly{k+1}(F)$ that satisfies \eqref{eq:trF:polynomial.consistency}. One of them is to define, for all $\underline{q}_F\in\uHgrad[F]$, $\trFtilde\underline{q}_F\in\Poly{k+1}(F)$ such that
  \[
  \int_F\trFtilde\underline{q}_F\DIV_F\bvec{v}_F
  = -\int_F\cGF\underline{q}_F\cdot\bvec{v}_F
  + \sum_{E\in\EF}\omega_{FE}\int_E q_{\partial F}(\bvec{v}_F\cdot\normal_{FE})
  \qquad\forall\bvec{v}_F\in\Roly{k+2}(F)^\perp.
  \]
  This relation defines $\trFtilde\underline{q}_F$ uniquely since $\DIV_F:\Roly{k+2}(F)^\perp\to\Poly{k+1}(F)$ is an isomorphism (see \eqref{eq:iso:DIV}). The consistency property \eqref{eq:trF:polynomial.consistency} for this reconstruction can be checked setting $\underline{q}_F=\uIgrad[F] q$ for $q\in\Poly{k+1}(F)$, invoking the polynomial consistency property \eqref{eq:GF.k.cons} to infer $\cGF\underline{q}_F=\cGF(\uIgrad[F]q)=\GRAD_F q$, using the fact that $q_{\partial F}=q_{|\partial F}$ as a consequence of the definition \eqref{eq:uIgradF} of $\uIgrad[F]$ together with the isomorphism \eqref{eq:iso.Pc} and $q_{|\partial F}\in\Poly[\rm c]{k+1}(\partial F)$, and applying the integration by parts formula \eqref{eq:ipp:div.grad.F}.
\end{remark}

\subsubsection{Vector potential (tangential vector trace)}\label{sec:trFt}

We next define the two-dimensional vector potential $\trFt:\uHrot\to\Poly{k}(F)^2$ such that, for all $\uvec{v}_F\in\uHrot$,
\begin{subequations} \label{eq:trFt}
  \begin{alignat}{2} \label{eq:trFt:Roly.k}
    \int_F\trFt\uvec{v}_F\cdot\VROT_F r_F
    &= \int_F\cCF\uvec{v}_F~r_F
    + \sum_{E\in\EF}\omega_{FE}\int_E v_{\partial F} r_F
    &\qquad&\forall r_F\in\Poly{0,k+1}(F),
    \\ \label{eq:trFt:Roly.k.perp}
    \int_F\trFt\uvec{v}_F\cdot\bvec{w}_F
    &= \int_F\bvec{v}_{\cvec{R},F}^\perp\cdot\bvec{w}_F
    &\qquad&\forall\bvec{w}_F\in\Roly{k}(F)^\perp.
  \end{alignat}
\end{subequations}
To check that, given $\uvec{v}_F\in\uHrot$, \eqref{eq:trFt} defines a unique polynomial $\trFt\uvec{v}_F\in\Poly{k}(F)^2$, observe that \eqref{eq:trFt:Roly.k} and \eqref{eq:trFt:Roly.k.perp} prescribe, respectively, $\Rproj{k}{F}(\trFt\uvec{v}_F)$ (use  \eqref{eq:iso:VROTF.GRAD}) and $\ROproj{k}{F}(\trFt\uvec{v}_F)$, and recall the orthogonal decomposition $\Poly{k}(F)^2=\Roly{k}(F)\oplus\Roly{k}(F)^\perp$.
When $F$ represents a face of a polyhedron $T$, $\trFt$ corresponds to a reconstruction of the tangential trace.

\begin{remark}[Validity of~\eqref{eq:trFt:Roly.k}]\label{rem:valid.eq:trFt:Roly.k}
  Observing that both sides of \eqref{eq:trFt:Roly.k} vanish for $r_F = 1$ (use the definition of $\cCF$
  for the right-hand side), we deduce that \eqref{eq:trFt:Roly.k} holds in fact for any $r_F\in\Poly{k+1}(F)$.
\end{remark}

\begin{remark}[Alternative reconstructions]\label{rem:trFt:alternative}
    Alternative vector potential reconstructions can also be considered in practical applications.
    If it is required, e.g., that the $L^2$-orthogonal projection of the potential reconstruction on $\Roly{k-1}(F)$ coincides with the internal unknown $\bvec{v}_{\cvec{R},F}$, one can take $\trFt\uvec{v}_F-\Rproj{k-1}{F}\trFt\uvec{v}_F+\bvec{v}_{\cvec{R},F}$ as two-dimensional vector potential, with $\trFt$ defined by \eqref{eq:trFt}.
\end{remark}

We now state and prove two propositions on the commutation properties of the vector potential reconstruction.

\begin{proposition}[Commutation property for the two-dimensional vector potential reconstruction]\label{prop:trFt:consistency.1}
  For all $\bvec{v}\in H^1(F)^2$ such that $\ROT_F\bvec{v}\in\Poly{k}(F)$ and $\bvec{v}_{|E}\cdot\tangent_E\in\Poly{k}(E)$ for all $E\in\EF$, it holds
  \begin{equation} \label{eq:trFt:polynomial.consistency}
    \trFt(\uIrot\bvec{v}) = \vlproj{k}{F}\bvec{v}.
  \end{equation}
\end{proposition}
\begin{proof}
  Take $r_F\in\Poly{k+1}(F)$, write \eqref{eq:trFt:Roly.k} (using Remark \ref{rem:valid.eq:trFt:Roly.k}) for $\uvec{v}_F=\uIrot\bvec{v}$, and apply the commutation property \eqref{eq:CF.commutation} of $\cCF$ to get
  \[
  \int_F\trFt(\uIrot \bvec{v})\cdot\VROT_F r_F = \int_F\lproj{k}{F}(\ROT_F \bvec{v})~r_F
  + \sum_{E\in\EF}\omega_{FE}\int_E \lproj{k}{E}(\bvec{v}_{|E}\cdot\tangent_E) r_F.
  \]
  Using the assumptions on $\bvec{v}$ along with their definition \eqref{eq:lproj}, the projectors $\lproj{k}{F}$ and $\lproj{k}{E}$ can be removed from the equation above, and the integration by parts formula \eqref{eq:ipp:rot.F} then leads to
  \[
  \int_F\trFt(\uIrot \bvec{v})\cdot\VROT_F r_F = \int_F \bvec{v}\cdot\VROT_F r_F.
  \]
  The polynomial $r_F$ being arbitrary in $\Poly{k+1}(F)$, this relation implies $\Rproj{k}{F}(\trFt(\uIrot \bvec{v}))=\Rproj{k}{F}\bvec{v}$.
  On the other hand, \eqref{eq:trFt:Roly.k.perp} with $\uvec{v}_F=\uIrot \bvec{v}$ and the definition of $\uIrot$ yield $\ROproj{k}{F}(\trFt(\uIrot \bvec{v}))=\bvec{v}_{\cvec{R},F}^\perp=\ROproj{k}{F}\bvec{v}$. The relation \eqref{eq:trFt:polynomial.consistency} then follows using the decomposition $\Poly{k}(F)^2=\Roly{k}(F)\oplus \Roly{k}(F)^\perp$ to write
  \[
  \trFt(\uIrot \bvec{v})
  = \Rproj{k}{F}(\trFt(\uIrot \bvec{v}))+\ROproj{k}{F}(\trFt(\uIrot \bvec{v}))
  = \Rproj{k}{F}\bvec{v}+\ROproj{k}{F}\bvec{v}
  = \vlproj{k}{F}\bvec{v}.\qedhere
  \]
\end{proof}

\begin{proposition}[Two-dimensional vector potential reconstruction and gradient]
  \label{prop:trFt:consistency.2}
  It holds
  \begin{equation} \label{eq:trFt:GF}
    \trFt(\uGF\underline{q}_F) = \cGF\underline{q}_F\qquad\forall \underline{q}_F\in\uHgrad[F].
  \end{equation}
\end{proposition}
\begin{proof}
  For all $r_F\in\Poly{k+1}(F)$, writing \eqref{eq:trFt:Roly.k} for $\uvec{v}_F=\uGF\underline{q}_F$, it is inferred that
  \[
  \int_F\trFt(\uGF\underline{q}_F)\cdot\VROT_F r_F
  = \int_F\cancel{\cCF(\uGF\underline{q}_F)}~r_F + \sum_{E\in\EF}\omega_{FE}\int_E\cGpF\underline{q}_F~r_F
  = \int_F\cGF\underline{q}_F\cdot\VROT_F r_F,
  \]  
  where we have used the inclusion \eqref{eq:Image.uGF.subset.Ker.CF} in the cancellation, while the conclusion follows proceeding as in \eqref{eq:exact.CF.middle}.
  This implies $\Rproj{k}{F}(\trFt(\uGF\underline{q}_F))=\Rproj{k}{F}(\cGF\underline{q}_F)$.
  On the other hand, \eqref{eq:trFt:Roly.k.perp} also applied to $\uvec{v}_F=\uGF\underline{q}_F$ implies $\ROproj{k}{F}(\trFt(\uGF\underline{q}_F))=\GROF\underline{q}_F=\ROproj{k}{F}(\cGF\underline{q}_F)$.
  Combining these relations with the orthogonal decomposition $\Poly{k}(F)^2 = \Roly{k}(F) \oplus \Roly{k}(F)^\perp$, \eqref{eq:trFt:GF} follows.
\end{proof}

\subsection{Two-dimensional discrete $L^2$-products}\label{sec:2d.sequence:L2.products}

We next define discrete counterparts of the $L^2$-products in $H^1(F)$ and $\Hrot{F}$.
The discrete $L^2$-products are composed of consistent and stabilising terms.
The former correspond to the $L^2$-product of the full potential reconstructions, whereas the latter penalise in a least square sense high-order differences between the potential reconstruction and the discrete unknowns. The design of these high-order differences is inspired by the stabilisation terms in HHO methods, see \cite[Section 2.1.4]{Di-Pietro.Droniou:20}.
Specifically, we define
\begin{itemize}
\item $(\cdot,\cdot)_{\GRAD,F}:\uHgrad[F]\times\uHgrad[F]\to\Real$ such that, for all $\underline{q}_F,\underline{r}_F\in\uHgrad[F]$,
  \begin{multline}\label{eq:().grad.F}
    (\underline{q}_F,\underline{r}_F)_{\GRAD,F}
    \coloneq
    \int_F\trF\underline{q}_F~\trF\underline{r}_F + \int_F \delta_{\GRAD,F}^{k-1}\underline{q}_F~ \delta_{\GRAD,F}^{k-1}\underline{r}_F
    \\
    + \sum_{E\in\EF} h_F\int_{E}\delta_{\GRAD,\partial F}^{k+1}\underline{q}_F~\delta_{\GRAD,\partial F}^{k+1}\underline{r}_F,
  \end{multline}
  where $h_F$ denotes the diameter of $F$ and we have set, for any $\underline{q}_F\in\uHgrad[F]$,
  \[
  (\delta_{\GRAD,F}^{k-1}\underline{q}_F,\delta_{\GRAD,\partial F}^{k+1}\underline{q}_F)
  \coloneq
  \uIgrad[F](\trF\underline{q}_F)-\underline{q}_F.
  \]
\item $(\cdot,\cdot)_{\ROT,F}:\uHrot\times\uHrot\to\Real$ such that, for all $\uvec{v}_F,\uvec{w}_F\in\uHrot$,
  \begin{multline}\label{eq:().rot.F}
    (\uvec{v}_F,\uvec{w}_F)_{\ROT,F}
    \coloneq
    \int_F\trFt\uvec{v}_F\cdot\trFt\uvec{w}_F+\int_F \bvec{\delta}_{\ROT,F}^{k-1}\uvec{v}_F\cdot\bvec{\delta}_{\ROT,F}^{k-1}\uvec{w}_F
    \\
    +\int_F \bvec{\delta}_{\ROT,F}^{\perp,k}\uvec{v}_F\cdot\bvec{\delta}_{\ROT,F}^{\perp,k}\uvec{w}_F+ \sum_{E\in\EF}h_F\int_{E}\delta_{\ROT,E}^k\uvec{v}_F~\delta_{\ROT,E}^k\uvec{w}_F,
  \end{multline}
  where we have set, for any $\uvec{v}_F\in\uHrot$,
  \[
  (\bvec{\delta}_{\ROT,F}^{k-1}\uvec{v}_F+\bvec{\delta}_{\ROT,F}^{\perp,k}\uvec{v}_F,(\delta_{\ROT,E}^k\uvec{v}_F)_{E\in\EF}) \coloneq \uIrot(\trFt\uvec{v}_F)- \uvec{v}_F.
  \]
\end{itemize}
The bilinear forms $(\cdot,\cdot)_{\GRAD,F}$ and $(\cdot,\cdot)_{\ROT,F}$ are obviously symmetric and positive semi-definite.
Using arguments similar to the ones deployed in the three-dimensional case (cf. Lemma \ref{lem:L2.products:3d} below), it can be proved that they are actually positive definite, hence they define proper inner products on $\uHgrad[F]$ and $\uHrot$, respectively.
By \eqref{eq:trF:polynomial.consistency} and \eqref{eq:trFt:polynomial.consistency}, they also enjoy the following consistency properties:
\begin{alignat*}{2}
  (\uIgrad[F]q,\uIgrad[F]r)_{\GRAD,F} &= (q,r)_{L^2(F)} &\qquad&\forall q,r\in\Poly{k+1}(F),\\
  (\uIrot\bvec{v},\uIrot\bvec{w})_{\ROT,F} &= (\bvec{v},\bvec{w})_{L^2(F)^2}&\qquad&\forall\bvec{v},\bvec{w}\in\Poly{k}(F)^2.
\end{alignat*}


\section{An exact three-dimensional sequence}\label{sec:3d.sequence}

In this section we define a discrete counterpart of the following exact three-dimensional sequence on a polyhedron $T$ (which may be thought of as a mesh element):
\[
\begin{tikzcd}
  \Real\arrow{r}{i_T} & H^1(T)\arrow{r}{\GRAD} & \Hcurl{T}\arrow{r}{\CURL} & \Hdiv{T}\arrow{r}{\DIV} & L^2(T)\arrow{r}{0} & \{0\},
\end{tikzcd}
\]
where $i_T$ is the operator that maps a real value to a constant function over $T$, $H^1(T)$ denotes the space of functions that are square integrable over $T$ along with their derivatives, $\Hcurl{T}\coloneq\left\{\bvec{v}\in L^2(T)^3\st\CURL\bvec{v}\in L^2(T)^3\right\}$, and $\Hdiv{T}\coloneq\left\{\bvec{v}\in L^2(T)^3\st\DIV\bvec{v}\in L^2(T)\right\}$.
The principle is, as in two dimensions, to start from reconstructions of vector operators in full polynomial spaces, and to project them on restricted domains/co-domains to form an exact sequence. For the sake of conciseness, and since it is similar to the two-dimensional case, we do not detail the initial analysis (i.e., the derivation of an almost-exact sequence and the equivalent of Proposition \ref{prop:almost.exact}), but directly provide appropriate choices of spaces and discrete operators.

\subsection{Three-dimensional discrete spaces and interpolators}\label{sec:3d.sequence:spaces}

\subsubsection[Discrete H1(T) space]{Discrete $H^1(T)$ space}

The discrete counterpart of $H^1(T)$ is
\begin{equation}\label{eq:uHgrad.T}
  \uHgrad\coloneq\Poly{k-1}(T)\times \Poly{k-1}(\FT)\times \Poly[c]{k+1}(\partial^2 T).
\end{equation}
A generic vector $\underline{q}_T\in\uHgrad$ is written 
\[
\underline{q}_T=(q_T,q_{\partial T},q_{\partial^2 T})\,,\text{ with $q_T\in\Poly{k-1}(T)$, $q_{\partial T}\in\Poly{k-1}(\FT)$ and $q_{\partial^2 T}\in\Poly[c]{k+1}(\partial^2 T)$}.
\]
For any $F\in\FT$ we let $q_F\coloneq (q_{\partial T})_{|F}$ and, for any $E\in\ET$, $q_E\coloneq(q_{\partial^2 T})_{|E}$.
The interpolator associated with this space is $\uIgrad:C^0(\overline{T})\to\uHgrad$ such that, for all $q\in C^0(\overline{T})$,
\begin{equation}\label{eq:uIgradT}
  \begin{gathered}
    \uIgrad q \coloneq \underline{q}_T=(q_T,q_{\partial T},q_{\partial^2 T})\in\uHgrad
    \text{ with
      $q_T=\lproj{k-1}{T} q$, $q_F=\lproj{k-1}{F}q_{|F}$ for all $F\in\FT$,}
    \\
    \text{$\lproj{k-1}{E} q_E=\lproj{k-1}{E} q_{|E}$ for all $E\in\ET$,
      and $q_{\partial^2 T}(\bvec{x}_V)=q(\bvec{x}_V)$ for all $V\in\VT$.
    }
  \end{gathered}
\end{equation}
Arguments similar to the two-dimensional case show that the component $q_{\partial^2 T}$ is well-defined by the conditions in the second line of \eqref{eq:uIgradT}.

We denote by $\uHgrad[\partial T]$ the restriction of $\uHgrad$ to $\partial T$, and the corresponding interpolator, with obvious definition, is denoted by $\uIgrad[\partial T]$.
Similarly, {$\uHgrad[\partial^2 T]$} is the restriction of $\uHgrad$ to $\partial^2 T$.
The restriction of $\underline{q}_T=(q_T,q_{\partial T},q_{\partial^2 T})\in\uHgrad$ to $\uHgrad[\partial T]$ is $\underline{q}_{\partial T}\coloneq(q_{\partial T},q_{\partial^2 T})$.
\begin{remark}[Relation with two-dimensional spaces]\label{rem:uHgrad.2D.3D}
  The restriction of an element $\underline{q}_T\in\uHgrad$ to a face $F\in\FT$ defines an element $\underline{q}_F\coloneq(q_F,(q_{\partial^2T})_{|\partial F})\in\uHgrad[F]$. Conversely, gluing together a family $(\underline{q}_F)_{F\in\FT}$ with $\underline{q}_F=(q_F,q_{\partial F})\in\uHgrad[F]$ for all $F\in\FT$ defines an element of $\uHgrad[\partial T]$ provided that the edge values coincide: For any $F_1,F_2$ faces of $T$ sharing an edge $E$, $(q_{\partial F_1})_{|E}=(q_{\partial F_2})_{|E}$.
\end{remark}

\subsubsection[Discrete H(curl;T) space]{Discrete $\Hcurl{T}$ space}

The role of the space $\Hcurl{T}$ is played, at the discrete level, by
\begin{equation}\label{eq:uHcurl.T}
  \uHcurl\coloneq \left(\Roly{k-1}(T)\oplus\Roly{k}(T)^\perp\right)\times \left(\bigtimes_{F\in\FT}\left[\Roly{k-1}(F)\oplus\Roly{k}(F)^\perp\right]\right)\times \Poly{k}(\ET).
\end{equation}
A generic vector $\uvec{v}_T\in\uHcurl$ is denoted by
\[
\begin{gathered}
  \uvec{v}_T=\big(
  \bvec{v}_T=\bvec{v}_{\cvec{R},T}+\bvec{v}_{\cvec{R},T}^\perp,(\bvec{v}_F=\bvec{v}_{\cvec{R},F}+\bvec{v}_{\cvec{R},F}^\perp)_{F\in\FT},v_{\partial^2 T}
  \big)
  \\
  \text{with $(\bvec{v}_{\cvec{R},T},\bvec{v}_{\cvec{R},T}^\perp)\in\Roly{k-1}(T)\times\Roly{k}(T)^\perp$,}
  \\
  \text{
    $(\bvec{v}_{\cvec{R},F},\bvec{v}_{\cvec{R},F}^\perp)\in\Roly{k-1}(F)\times\Roly{k}(F)^\perp$ for all $F\in\FT$, and
    $v_{\partial^2 T}\in\Poly{k}(\ET)$.
  }
\end{gathered}
\]
The interpolator $\uIcurl:C^0(\overline{T})^3\to\uHcurl$ is such that, for all $\bvec{v}\in C^0(\overline{T})^3$,
\begin{multline}\label{eq:uIcurlT}
  \uIcurl\bvec{v}
  \coloneq \\
  \big(
  \Rproj{k-1}{T}\bvec{v} + \ROproj{k}{T}\bvec{v},
  \big(
  \Rproj{k-1}{F}(\normal_F{\times}(\bvec{v}_{|F}{\times}\normal_F)) + \ROproj{k}{F}(\normal_F{\times}(\bvec{v}_{|F}{\times}\normal_F)
  \big)_{F\in\FT}, 
  \big(\lproj{k}{E}(\bvec{v}\cdot\tangent_E)\big)_{E\in\ET}
  \big).
\end{multline}
We remind the reader that $\normal_F\times(\bvec{v}_{|F}\times\normal_F)$ is the orthogonal projection of $\bvec{v}$ on the plane spanned by $F$.
The restriction of $\uHcurl$ to the boundary of $T$ is
\begin{multline*}
  \uHcurl[\partial T]\coloneq
  \Big\{
  \uvec{v}_{\partial T}\coloneq\big(
  (\bvec{v}_F=\bvec{v}_{\cvec{R},F}+\bvec{v}_{\cvec{R},F}^\perp)_{F\in\FT}, v_{\partial^2 T}
  \big)\st
  \\
  \text{
    $(\bvec{v}_{\cvec{R},F},\bvec{v}_{\cvec{R},F}^\perp)\in\Roly{k-1}(F)\times\Roly{k}(F)^\perp$ for all $F\in\FT$ and
    $v_{\partial^2T}\in\Poly{k}(\ET)$
  }
  \Big\}.
\end{multline*}

\begin{remark}[Relation with two-dimensional spaces]\label{rem:uHcurl.2D.3D}
  The restriction of an element $\uvec{v}_T\in\uHcurl$ to a face $F\in\FT$ defines an element $\uvec{v}_F\coloneq(\bvec{v}_F,(v_{\partial^2T})_{|\partial F})\in\uHrot$.
  Conversely, gluing together a family $(\uvec{v}_F)_{F\in\FT}$ with $\uvec{v}_F=(\bvec{v}_F,v_{\partial F})\in\uHrot$ for all $F\in\FT$ defines an element of $\uHcurl[\partial T]$ provided that the edge values coincide: For any $F_1,F_2$ faces of $T$ sharing an edge $E$, $(v_{\partial F_1})_{|E}=(v_{\partial F_2})_{|E}$.
\end{remark}

\subsubsection[Discrete H(div;T) space]{Discrete $\Hdiv{T}$ space}

Finally, the discrete counterpart of the space $\Hdiv{T}$ is
\begin{equation}\label{eq:uHdiv.T}
  \uHdiv\coloneq \left(\Goly{k-1}(T)\oplus\Goly{k}(T)^\perp\right)\times\Poly{k}(\FT),
\end{equation}
with a generic vector $\uvec{v}_T\in\uHdiv$ decomposed as
\[
\text{
  $\uvec{v}_T\coloneq(\bvec{v}_T=\bvec{v}_{\cvec{G},T}+\bvec{v}_{\cvec{G},T}^\perp, v_{\partial T})$
  with $(\bvec{v}_{\cvec{G},T},\bvec{v}_{\cvec{G},T}^\perp)\in\Goly{k-1}(T)\times\Goly{k}(T)^\perp$ and $v_{\partial T}\in\Poly{k}(\FT)$.
}
\]
The interpolator $\uIdiv:C^0(\overline{T})^3\to\uHdiv$ is such that, for all $\bvec{v}\in C^0(\overline{T})^3$,
\begin{equation}\label{eq:uIdivT}
  \uIdiv\bvec{v}\coloneq\big(
  \Gproj{k-1}{T}\bvec{v}+\GOproj{k}{T}\bvec{v}, 
  \big(\lproj{k}{F}(\bvec{v}\cdot\normal_F)\big)_{F\in\FT}
  \big).
\end{equation}

\subsection{Three-dimensional vector operators reconstructions}

\subsubsection{Gradient}

The three-dimensional full gradient operator $\cGT:\uHgrad\to\Poly{k}(T)^3$ is defined such that, for all $\underline{q}_T\in\uHgrad$:
\begin{equation}\label{eq:GT.k}
  \int_T\cGT\underline{q}_T\cdot\bvec{v}_T
  = -\int_T q_T\DIV\bvec{v}_T
  + \sum_{F\in\FT}\omega_{TF}\int_F\trF\underline{q}_{\partial T}(\bvec{v}_T\cdot\normal_F)
  \qquad\forall\bvec{v}_T\in\Poly{k}(T)^3,
\end{equation}
where $(\trF:\uHgrad[\partial T]\to\Poly{k+1}(F))_{F\in\FT}$ is a family of trace reconstruction operators such that, for all $F\in\FT$, $\trF$ only depends on the unknowns on $F$ and satisfies the polynomial consistency and projection properties \eqref{eq:trF:prop}.

The discrete gradient operator $\uGT:\uHgrad\to\uHcurl$ is defined by projecting $\cGT$ onto $\Roly{k-1}(T)\oplus\Roly{k}(T)^\perp$, and by completing with face and edge components obtained from the two-dimensional discrete gradient operators.
Specifically, for all $\underline{q}_T\in\uHgrad$, we let
\begin{equation}\label{eq:def.uGT}
  \begin{gathered}
    \uGT\underline{q}_T\coloneq \big(
    \GRT\underline{q}_T+\GROT\underline{q}_T,(\GRF\underline{q}_T+\GROF\underline{q}_T)_{F\in\FT},(\cGE\underline{q}_T)_{E\in\ET}
    \big)\,,\\
    \text{
      with $\GRT\coloneq \Rproj{k-1}{T}\cGT$, $\GROT\coloneq\ROproj{k}{T}\cGT$,
    }
    \\
    \text{
      and $\GRF$, $\GROF$, and $\cGE$ formally defined by \eqref{eq:def.GF.GOF} and \eqref{eq:GE.k}, respectively.
    }
  \end{gathered}
\end{equation}
Accounting for the fact that the face and edge gradients depend only on boundary unknowns, with a little abuse of notation we will use the same symbols for the operators obtained by restricting their domain: $\GRF:\uHgrad[\partial T]\to\Roly{k-1}(F)$, $\GROF:\uHgrad[\partial T]\to\Roly{k}(F)^\perp$, $\cGE:\uHgrad[\partial T]\to\Poly{k}(E)$, and $\cGE:\uHgrad[\partial^2 T]\to\Poly{k}(E)$.
\subsubsection{Curl}

The full curl operator $\cCT:\uHcurl\to\Poly{k}(T)^3$ is defined such that, for all $\uvec{v}_T\in\uHcurl$,
\begin{equation}\label{eq:CT.k}
  \int_T\cCT\uvec{v}_T\cdot\bvec{w}_T
  = \int_T\bvec{v}_T\cdot\CURL\bvec{w}_T
  + \sum_{F\in\FT}\omega_{TF}\int_F\trFt\uvec{v}_{\partial T}\cdot(\bvec{w}_T\times\normal_F)
  \qquad\forall\bvec{w}_T\in\Poly{k}(T)^3,
\end{equation}
where $(\trFt:\uHcurl[\partial T]\to\Poly{k}(F)^2)_{F\in\FT}$ is the family of tangential trace reconstruction operators such that, for all $F\in\FT$, $\trFt$ is formally defined as in Section \ref{sec:trFt} (see, in particular, \eqref{eq:trFt}).
The discrete curl operator $\uCT:\uHcurl\to\uHdiv$ is obtained projecting $\cCT$ onto $\Goly{k-1}(T)\oplus\Goly{k}(T)^\perp$ and completing using the face curl operators:
For all $\uvec{v}_T\in\uHcurl$,
\begin{equation}\label{eq:def.uCT}
  \begin{gathered}
    \uCT\uvec{v}_T \coloneq \big(\CGT\uvec{v}_T+\CGOT\uvec{v}_T,(\cCF\uvec{v}_T)_{F\in\FT}\big)
    \text{ with $\CGT\coloneq\Gproj{k-1}{T}\cCT$, $\CGOT\coloneq \GOproj{k}{T}\cCT$,}\\
    \text{ and $\cCF:\uHcurl\to\Poly{k}(F)$ formally defined by \eqref{eq:cCF}.
    }
  \end{gathered}
\end{equation}

\subsubsection{Divergence}

The discrete and full divergence operators are both equal to $\cDT:\uHdiv\to\Poly{k}(T)$ defined such that, for all $\uvec{v}_T\in\uHdiv$,
\begin{equation}\label{eq:DT}
  \int_T\cDT\uvec{v}_T~ q_T
  = -\int_T\bvec{v}_T\cdot\GRAD q_T
  + \sum_{F\in\FT}\omega_{TF}\int_F v_F q_T
  \qquad\forall q_T\in\Poly{k}(T).
\end{equation}

\subsection{Exactness of the three-dimensional sequence}

The goal of this section is to prove the following exactness result.

\begin{theorem}[Exact three-dimensional sequence]\label{th:exact.3D}
  The following sequence is exact:
  \begin{equation}\label{eq:sequence.3D}
    \begin{tikzcd}
      \Real\arrow{r}{\uIgrad} & \uHgrad\arrow{r}{\uGT} & \uHcurl\arrow{r}{\uCT} & \uHdiv\arrow{r}{\cDT} & \Poly{k}(T)\arrow{r}{0} & \{0\}.
    \end{tikzcd}
  \end{equation}
\end{theorem}

\begin{remark}[Comparison with Finite Elements]
  When $T$ is a tetrahedron or a hexahedron, the sequence \eqref{eq:sequence.3D} can be compared with the usual FE sequences.
  The corresponding number of discrete unknowns is reported in Table \ref{tab:comparison.3D}.
  Similar considerations as for the two-dimensional case hold; see Remark \ref{rem:comparison.2D}.
\end{remark}

\begin{table}\centering
  \begin{small}
    \begin{tabular}{c|c|c|c|cc|cc|cc}
      \toprule
      & \multirow{2}{*}{$k$} & \multirow{2}{*}{$V\in\VT$} & \multirow{2}{*}{$E\in\ET$} & \multicolumn{2}{c|}{$F\in\FT$} & \multicolumn{2}{c|}{$T$} & \multicolumn{2}{c}{Total}
      \\
      & & & & Tetra & Hexa & Tetra & Hexa & Tetra & Hexa
      \\\midrule
      \multirow{4}{*}{$\uHgrad$} 
      & 0 & \cellcolor{black!10}{1} & \cellcolor{black!10}{0} & 0 (0) & 0 (0) & 0 (0) & 0 (0) & 4 (4) & 8 (8)
      \\
      & 1 & \cellcolor{black!10}{1} & \cellcolor{black!10}{1} & 1 (0) & 1 (1) & 1 (0) & 1 (1) & 15 (10) & 27 (27)
      \\
      & 2 & \cellcolor{black!10}{1} & \cellcolor{black!10}{2} & 3 (1) & 3 (4) & 4 (0) & 4 (8) & 32 (20) & 54 (64)
      \\
      & 3 & \cellcolor{black!10}{1} & \cellcolor{black!10}{3} & 6 (3) & 6 (9) & 10 (1) & 10 (27) & 56 (35) & 90 (125)
      \\\midrule
      \multirow{4}{*}{$\uHcurl$}
      & 0 & & \cellcolor{black!10}{1} & 0 (0) & 0 (0) & 0 (0) & 0 (0) & 6 (6) & 12 (12)
      \\
      & 1 & & \cellcolor{black!10}{2} & 3 (2) & 3 (4) & 4 (0) & 4 (6) & 28 (20) & 46 (54)
      \\
      & 2 & & \cellcolor{black!10}{3} & 8 (6) & 8 (12) & 15 (3) & 15 (36) & 65 (45) & 99 (144)
      \\
      & 3 & & \cellcolor{black!10}{4} & 15 (12) & 15 (24) & 36 (12) & 36 (72) & 120 (84) & 174 (264)
      \\\midrule
      \multirow{4}{*}{$\uHdiv$}
      & 0 & & & \cellcolor{black!10}{1} & 1 (1) & 0 (0) & 0 (0) & 4 (4) & 6 (6)
      \\
      & 1 & & & \cellcolor{black!10}{3} & 3 (4) & 6 (3) & 6 (12) & 18 (15) & 24 (36)
      \\
      & 2 & & & \cellcolor{black!10}{6} & 6 (9) & 20 (12) & 20 (54) & 44 (36) & 56 (108)
      \\
      & 3 & & & \cellcolor{black!10}{10} & 10 (16) & 45 (30) & 45 (144) & 85 (70) & 105 (240)
      \\\midrule
      \multirow{4}{*}{$\Poly{k}(T)$}
      & 0 & & & & & \cellcolor{black!10}{1} & 1 (1) & \cellcolor{black!10}{1} & 1 (1) 
      \\
      & 1 & & & & & \cellcolor{black!10}{4} & 4 (8) & \cellcolor{black!10}{4} & 4 (8)
      \\
      & 2 & & & & & \cellcolor{black!10}{10} & 10 (27) & \cellcolor{black!10}{10} & 10 (27)
      \\
      & 3 & & & & & \cellcolor{black!10}{20} & 20 (64) &  \cellcolor{black!10}{20} & 20 (64)
      \\\bottomrule
    \end{tabular}
  \end{small}
  \caption{Number of discrete unknowns attached to each geometric entity for the three-dimensional sequence \eqref{eq:sequence.3D} on a tetrahedron or hexahedron $T$ for $k\in\{0,\ldots,3\}$.
    For comparison, we also report in parentheses the number of DOFs of the spaces in the corresponding FE sequence when the latter are different.
    Gray cells highlight the cases for which, irrespectively of the degree $k$, the number of DOFs are the same for the three-dimensional sequence \eqref{eq:sequence.3D} and for the FE sequences.
    \label{tab:comparison.3D}}  
\end{table}

\subsubsection{Preliminary results}

We establish first a few properties on the discrete operators that will be useful during the proof of Theorem \ref{th:exact.3D}. We start by noticing the following relations:
\begin{alignat}{2}
  \label{eq:cancel.edges}
  &\sum_{F\in\FT}\sum_{E\in\EF}\omega_{TF}\omega_{FE}a_E
  =0&\qquad&\forall (a_E)_{E\in\ET}\in\Real^{\ET},\\
  \label{eq:righthanded}
  &(\bvec{z}\times \normal_F)\cdot\normal_{FE}
  =\bvec{z}\cdot\tangent_E&\qquad&\forall \bvec{z}\in\Real^3\,,\;\forall F\in\FT\,,\;\forall E\in\EF.
\end{alignat}
Relation \eqref{eq:cancel.edges} follows from the fact that, after rearranging the sum over the edges, each $a_E$ appears in factor of the quantity in the left-hand side of \eqref{eq:orientation}. Equation \eqref{eq:righthanded} is a direct consequence of $(\bvec{z}\times \normal_F)\cdot\normal_{FE}=-\bvec{z}\cdot(\normal_F\times\normal_{FE})$ together with the fact that $(\tangent_E,\normal_{FE},\normal_F)$ is a right-handed system of coordinates.


\begin{lemma}[Properties of the gradient operator]\label{lem:properties.gradient}
  It holds:
  \begin{equation}
    \label{eq:GT:polynomial.consistency}
    \cGT(\uIgrad q) = \GRAD q \qquad\forall q\in\Poly{k+1}(T),
  \end{equation}
  and
  \begin{multline}\label{eq:link.GT.GF}
    \int_T\GRT\underline{q}_T\cdot\CURL\bvec{w}_T
    = \int_T\cGT\underline{q}_T\cdot\CURL\bvec{w}_T
    = -\sum_{F\in\FT}\omega_{TF}\int_F \cGF\underline{q}_{F}\cdot(\bvec{w}_T\times\normal_F)
    \\
    \forall\underline{q}_T\in\uHgrad\,,\;\forall\bvec{w}_T\in\Poly{k}(T)^3.
  \end{multline}
\end{lemma}

\begin{remark}[Properties \eqref{eq:GT:polynomial.consistency}--\eqref{eq:link.GT.GF}]
  The relation \eqref{eq:GT:polynomial.consistency} states a consistency property on the full gradient reconstruction, 
  while \eqref{eq:link.GT.GF} provides a link between volume and face gradients.
\end{remark}

\begin{proof}[Proof of Lemma \ref{lem:properties.gradient}]
  \underline{1. Proof of \eqref{eq:GT:polynomial.consistency}}.
  Let $q\in\Poly{k+1}(T)$ and apply the definition \eqref{eq:GT.k} of $\cGT$ to $\underline{q}_T\coloneq\uIgrad q$. Using the definition \eqref{eq:uIgradT} of $\uIgrad$ together with the consistency \eqref{eq:trF:polynomial.consistency} of each $\trF$, this gives, for all $\bvec{v}_T\in\Poly{k}(T)^3$,
  \[
  \int_T \cGT(\uIgrad q)\cdot\bvec{v}_T = -\int_T \lproj{k-1}{T}q~\DIV\bvec{v}_T
  + \sum_{F\in\FT}\omega_{TF}\int_F q(\bvec{v}_T\cdot\normal_F)\\
  =\int_T \GRAD q\cdot\bvec{v}_T,
  \]
  where the second equality is obtained removing the projector $\lproj{k-1}{T}$ (since $\DIV\bvec{v}_T\in\Poly{k-1}(T)$) and using the integration by parts formula \eqref{eq:ipp:div.grad.T}.
  Since both $\cGT(\uIgrad q)$ and $\GRAD q$ belong to $\Poly{k}(T)^3$, this relation establishes \eqref{eq:GT:polynomial.consistency}.
  \medskip\\
  \underline{2. Proof of \eqref{eq:link.GT.GF}}.
  The first equality is a straightforward consequence of $\GRT=\Rproj{k-1}{T}\cGT$ (see \eqref{eq:def.uGT}) and of $\CURL\bvec{w}_T\in\Roly{k-1}(T)$. To prove the second equality, apply the definition \eqref{eq:GT.k} of $\cGT$ to $\bvec{v}_T=\CURL\bvec{w}_T\in\Poly{k-1}(T)^3$ and introduce, using its definition, the projector $\lproj{k-1}{F}$ in each boundary integral to get
  \[
  \int_T\cGT\underline{q}_T\cdot\CURL\bvec{w}_T=-\int_T q_T~\cancel{\DIV\CURL \bvec{w}_T}+\sum_{F\in\FT}
  \omega_{TF}\int_F \lproj{k-1}{F}(\trF\underline{q}_{\partial T})(\CURL \bvec{w}_T\cdot\normal_F).
  \]
  Using the projection property \eqref{eq:trF:projection} of $\trF$ and the identity \eqref{eq:curlT.divF.rotF}, we infer
  \begin{multline}\label{eq:link.GT.GF.0}
    \int_T\cGT\underline{q}_T\cdot\CURL\bvec{w}_T
    =\sum_{F\in\FT}
    \omega_{TF}\int_F q_F \DIV_F(\bvec{w}_{T|F}\times\normal_F)
    \\
    =
    -\sum_{F\in\FT}\omega_{TF}\int_F\cGF\underline{q}_F\cdot(\bvec{w}_{T|F}\times\normal_F)
    +\sum_{F\in\FT}\sum_{E\in\EF}\omega_{TF}\omega_{FE}\int_E q_E(\bvec{w}_{T|F}\times\normal_F)\cdot\normal_{FE},
  \end{multline}
  where the second equality follows applying the definition \eqref{eq:GF.k} of $\cGF$ to $\bvec{w}_F \coloneq\bvec{w}_{T|F}\times\normal_F$. Using \eqref{eq:righthanded} we have $(\bvec{w}_{T|F}\times\normal_F)\cdot\normal_{FE}=\bvec{w}_{T}\cdot\tangent_E$ on each $E\in\ET$, and the double sum in \eqref{eq:link.GT.GF.0} therefore vanishes owing to \eqref{eq:cancel.edges} with $a_E=\int_Eq_E(\bvec{w}_T\cdot\tangent_E)$. The proof of the second equality in \eqref{eq:link.GT.GF} is complete.
\end{proof}


The following N\'ed\'elec space, in which $\overline{\bvec{x}}_T\coloneq\frac{1}{|T|}\int_T\bvec{x}\ud\bvec{x}$ denotes the centroid of $T$, will be useful to formulate a commutation property for $\cCT$:
\begin{equation}\label{def:nedelec}
  \NE{k}(T) \coloneq \Poly{k}(T)^3 + (\bvec{x}-\overline{\bvec{x}}_T)\times\Poly{k}(T)^3.
\end{equation}

\begin{lemma}[Properties of the curl operator]\label{lem:prop.CT}
  It holds:
  \begin{alignat}{2}\label{eq:CT.k:polynomial.consistency}
    \cCT(\uIcurl\bvec{v})&=\CURL\bvec{v}&\qquad&\forall\bvec{v}\in\NE{k}(T),\\
    \label{eq:rep.CGT}
    \int_T\CGT\uvec{v}_T\cdot\bvec{w}_T
    &=\sum_{F\in\FT}\omega_{TF}\int_F\bvec{v}_F\cdot(\bvec{w}_T\times\normal_F)
    &\qquad&\forall\uvec{v}_T\in\uHcurl\,,\;\forall\bvec{w}_T\in\Goly{k-1}(T),\\
    \label{eq:CT.k.CF.k}
    \int_T\cCT\uvec{v}_T\cdot\GRAD r_T
    &=  \sum_{F\in\FT}\omega_{TF}\int_F\cCF\uvec{v}_{\partial T}~r_T
    &\qquad&\forall \uvec{v}_T\in\uHcurl\,,\;\forall r_T\in\Poly{k+1}(T),\\
    \label{eq:CT.CF}
    \int_T \CGT\uvec{v}_T\cdot\GRAD r_T
    &=  \sum_{F\in\FT}\omega_{TF}\int_F\cCF\uvec{v}_{\partial T}~r_T
    &\qquad&\forall \uvec{v}_T\in\uHcurl\,,\;\forall r_T\in\Poly{k}(T).
  \end{alignat}
\end{lemma}

\begin{remark}[Properties \eqref{eq:CT.k:polynomial.consistency}--\eqref{eq:CT.CF}]
  Equation \eqref{eq:CT.k:polynomial.consistency} states a polynomial consistency property for $\cCT$, 
  \eqref{eq:rep.CGT} is a characterisation of $\CGT$, and \eqref{eq:CT.k.CF.k}--\eqref{eq:CT.CF} provide links between volume and face curls.
\end{remark}

\begin{proof}[Proof of Lemma \ref{lem:prop.CT}]
  \underline{1. Proof of \eqref{eq:CT.k:polynomial.consistency}}.
  Let $\bvec{v}\in \NE{k}(T)$ and set $\uvec{v}_T\coloneq\uIcurl\bvec{v}$.
  The function $\bvec{v}$ belongs to $\Poly{k+1}(T)^3$ and thus $\normal_F \times (\bvec{v}_{|F}\times \normal_F)\in\Poly{k+1}(F)^2$ for all $F\in\FT$. As a consequence, $\ROT_F (\normal_F \times (\bvec{v}_{|F}\times \normal_F))\in\Poly{k}(F)$.
  Write now $\bvec{v}=\bvec{w}+(\bvec{x}-\overline{\bvec{x}}_T)\times \bvec{z}$ with $\bvec{w},\bvec{z}\in\Poly{k}(T)^3$, take $E\in\EF$, and denote by $\overline{\bvec{x}}_E$ the midpoint of $E$. For all $\bvec{x}\in E$, we have
  \begin{align*}
    (\normal_F \times (\bvec{v}(\bvec{x})\times \normal_F))\cdot\tangent_E=\bvec{v}(\bvec{x})\cdot\tangent_E
    ={}&\bvec{w}(\bvec{x})\cdot\tangent_E + \left( (\bvec{x}-\overline{\bvec{x}}_T)\times \bvec{z}(\bvec{x}) \right)\cdot\tangent_E\\
    ={}&\bvec{w}(\bvec{x})\cdot\tangent_E + \cancel{\left( (\bvec{x}-\overline{\bvec{x}}_E)\times \bvec{z}(\bvec{x}) \right)\cdot\tangent_E} + \left( (\overline{\bvec{x}}_E-\overline{\bvec{x}}_T)\times\bvec{z}(\bvec{x}) \right)\cdot\tangent_E,
  \end{align*}
  the first equality coming from the fact that $\normal_F \times (\bvec{v}(\bvec{x})\times \normal_F)$ is the projection of $\bvec{v}(\bvec{x})$ on the plane spanned by $F$, and the cancellation in the second line being justified by the fact that $\bvec{x}-\overline{\bvec{x}}_E$ is parallel to $\tangent_E$. Since both $\bvec{w}$ and $\bvec{z}$ are polynomials of degree $\le k$, this proves that $(\normal_F \times (\bvec{v}_{|E}\times \normal_F))\cdot\tangent_E \in\Poly{k}(E)$.
  We have thus shown that $\normal_F \times (\bvec{v}_{|F}\times \normal_F)$ satisfies the assumptions in Proposition \ref{prop:trFt:consistency.1} and, noticing that $\uvec{v}_{\partial T}$ is given on each face $F\in\FT$ by $\uIrot[F](\normal_F \times (\bvec{v}_{|F}\times \normal_F))$, we infer that
  \begin{equation}\label{eq:comut.cCT.0}
    \trFt\uvec{v}_{\partial T}=\vlproj{k}{F}(\normal_F\times(\bvec{v}_{|F}\times\normal_F))\qquad\forall F\in\FT.
  \end{equation}
  By definition of $\uIcurl$, we have $\bvec{v}_T=\Rproj{k-1}{T}\bvec{v} + \ROproj{k}{T}\bvec{v}$, and thus, for any $\bvec{w}_T\in\Poly{k}(T)^3$,
  \[
  \int_T \bvec{v}_T\cdot\CURL\bvec{w}_T=\int_T(\Rproj{k-1}{T}\bvec{v} + \ROproj{k}{T}\bvec{v})\cdot\CURL\bvec{w}_T=
  \int_T \bvec{v}\cdot\CURL\bvec{w}_T,
  \]
  where we have removed $\Rproj{k-1}{T}$ using its definition, and $\ROproj{k}{T}$ using its $L^2$-orthogonality to $\CURL\bvec{w}_T\in\Roly{k-1}(T)\subset\Roly{k}(T)$.
    Hence, applying the definition \eqref{eq:CT.k} of $\cCT$ and using \eqref{eq:comut.cCT.0}, we obtain, for all $\bvec{w}_T\in\Poly{k}(T)^3$,
  \begin{align*}
    \int_T\cCT(\uIcurl \bvec{v})\cdot\bvec{w}_T={}&\int_T\bvec{v}\cdot\CURL\bvec{w}_T
    +\sum_{F\in\FT}\omega_{TF}\int_F\vlproj{k}{F}\left(\normal_F\times(\bvec{v}_{|F}\times\normal_F)\right)\cdot(\bvec{w}_T\times\normal_F)\\
    ={}&\int_T\bvec{v}\cdot\CURL\bvec{w}_T    
    +\sum_{F\in\FT}\omega_{TF}\int_F\left(\normal_F\times(\bvec{v}_{|F}\times\normal_F)\right)\cdot(\bvec{w}_T\times\normal_F)
    \\
    ={}&\int_T \CURL\bvec{v}\cdot\bvec{w}_T,
  \end{align*}
  where the second line follows from $\bvec{w}_{T|F}\times\normal_F\in\Poly{k}(F)^2$, and the conclusion is a consequence of the integration by parts formula \eqref{eq:ipp:curl.T}. Since both $\cCT(\uIcurl \bvec{v})$ and $\CURL \bvec{v}$ belong to $\Poly{k}(T)^3$, this concludes the proof of \eqref{eq:CT.k:polynomial.consistency}.
  \medskip\\
  \underline{2. Proof of \eqref{eq:rep.CGT}}. Recalling that $\CGT=\Gproj{k-1}{T}\cCT$ (see \eqref{eq:def.uCT}) and using the definition \eqref{eq:CT.k} of $\cCT$ we infer, for all $\bvec{w}_T\in\Goly{k-1}(T)$,
  \[
  \int_T\CGT\uvec{v}_T\cdot\bvec{w}_T=  \int_T\cCT\uvec{v}_T\cdot\bvec{w}_T=\int_T \bvec{v}_T\cdot\cancel{\CURL\bvec{w}_T}+\sum_{F\in\FT}\omega_{TF}\int_F \trFt\uvec{v}_{\partial T}\cdot(\bvec{w}_T\times\normal_F),
  \]
  the cancellation coming from the vector calculus identity $\CURL\GRAD=\bvec{0}$.
  Let now $r_T\in\Poly{k}(T)$ be such that $\bvec{w}_T=\GRAD r_T$. Then the identity \eqref{eq:gradT.rotF} yields $\bvec{w}_T\times\normal_F=\VROT_F r_{T|F}$ and thus
  \begin{align*}
    \int_T\CGT\uvec{v}_T\cdot\bvec{w}_T={}& \sum_{F\in\FT}\omega_{TF}\int_F \trFt\uvec{v}_{\partial T}\cdot\VROT_F r_{T|F}\\
    ={}& \sum_{F\in\FT}\omega_{TF}\left(\int_F \cCF\uvec{v}_{\partial T}~r_{T|F}
    +\sum_{E\in\EF}\omega_{FE}\int_E v_{\partial^2 T}~r_{T|F}\right)\\
    ={}&\sum_{F\in\FT}\omega_{TF}\int_F \bvec{v}_F\cdot\VROT_F r_{T|F},
  \end{align*}
  where the second line follows from the definition \eqref{eq:trFt:Roly.k} of $\trFt$ applied to $r_F=r_{T|F}\in\Poly{k}(F)\subset\Poly{k+1}(F)$ (see also Remark \ref{rem:valid.eq:trFt:Roly.k}) and the third line is a consequence of the definition \eqref{eq:cCF} of $\cCF$ with the same $r_F$ (together with the definition $\bvec{v}_F=\bvec{v}_{\cvec{R},F}+\bvec{v}_{\cvec{R},F}^\perp$ and the $L^2$-orthogonality of $\bvec{v}_{\cvec{R},F}^\perp$ and $\VROT_F r_F$). The relation \eqref{eq:rep.CGT} follows by recalling that $\VROT_F r_{T|F}=\bvec{w}_T\times\normal_F$.
  \medskip\\
  \underline{3. Proof of \eqref{eq:CT.k.CF.k}--\eqref{eq:CT.CF}}. Let $\uvec{v}_T\in\uHcurl$ and $r_T\in\Poly{k+1}(T)$.
  Writing the definition \eqref{eq:CT.k} of $\cCT$ with $\bvec{w}_T=\GRAD r_T$, observing that $\CURL\GRAD r_T=\bvec{0}$, and using the identity \eqref{eq:gradT.rotF}, we infer
  \[
  \int_T\cCT\uvec{v}_T\cdot\GRAD r_T
  = \sum_{F\in\FT}\omega_{TF}\int_F\trFt\uvec{v}_{\partial T}\cdot\VROT_F r_{T|F}.
  \]
  We continue using, for all $F\in\FT$, the definition \eqref{eq:trFt:Roly.k} of the tangential trace reconstruction with $r_F=r_{T|F}\in\Poly{k+1}(F)$ (see also Remark \ref{rem:valid.eq:trFt:Roly.k}) to write
  \[
  \int_T\cCT\uvec{v}_T\cdot\GRAD r_T
  = \sum_{F\in\FT}\omega_{TF}\int_F\cCF\uvec{v}_{\partial T} r_T
  + \cancel{\sum_{F\in\FT}\omega_{TF}\sum_{E\in\EF}\omega_{FE}\int_Ev_Er_T},
  \]
  where, to cancel the rightmost term, we have used \eqref{eq:cancel.edges} with $a_E=\int_E v_E r_T$.
  This proves \eqref{eq:CT.k.CF.k}. The relation \eqref{eq:CT.CF} is obtained applying \eqref{eq:CT.k.CF.k} with $r_T\in\Poly{k}(T)$ and using $\CGT=\Gproj{k-1}{T}\cCT$.
\end{proof}

\begin{lemma}[Surjectivity of the boundary curl]\label{lem:surj.CurlF}
  Let $v_{\partial T}=(v_F)_{F\in\FT}\in\Poly{k}(\FT)$ be such that 
  \begin{equation}\label{eq:comp.surj.curl.dT}
    \sum_{F\in\FT}\omega_{TF}\int_{F}v_{F}=0.
  \end{equation}
  Then, there exists $\uvec{z}_{\partial T}\in \uHcurl[\partial T]$ such that $\cCF\uvec{z}_{\partial T}=v_F$ for all $F\in\FT$.
\end{lemma}

\begin{remark}
  The condition \eqref{eq:comp.surj.curl.dT} is also necessary for the conclusion of the lemma to hold.
\end{remark}

\begin{proof}[Proof of Lemma \ref{lem:surj.CurlF}]
  By the exactness in two dimensions (cf. Theorem \ref{th:exact.2D}), for each $F\in\FT$ there is $\uvec{w}_F\in\uHrot$ such that $\cCF\uvec{w}_F=v_F$. Following Remark \ref{rem:uHcurl.2D.3D}, a vector $\uvec{z}_{\partial T}\in \uHcurl[\partial T]$ can be defined by gluing together the vectors $(\uvec{w}_F)_{F\in\FT}$ if their edge values coincide. The idea of the proof is to add to each $\uvec{w}_F$ a vector $\uvec{y}_F\in \uHrot$ such that $\cCF(\uvec{w}_F+\uvec{y}_F)=v_F$ and the edge values of $(\uvec{w}_F+\uvec{y}_F)_{F\in\FT}$ coincide; by gluing these vectors together, we obtain $\uvec{z}_{\partial T}\in\uHcurl[\partial T]$ such that $\cCF\uvec{z}_{\partial T}=v_F$ for all $F\in\FT$.
  \smallskip
  
  To ensure the relation $\cCF(\uvec{w}_F+\uvec{y}_F)=v_F$, we have to look for $\uvec{y}_F$ in $\Ker\cCF$, that is, owing to Theorem \ref{th:exact.2D}, $\uvec{y}_F=\uGF \underline{q}_F$ for some $\underline{q}_F=(q_F,q_{\partial F})\in\uHgrad[F]$.
  The other condition on $\uvec{y}_F$ only concerns its boundary values $y_{\partial F}=\cGpF q_{\partial F}$ (which means that the component $q_F$ of $\underline{q}_F$ is irrelevant to our purpose and can be set to $0$), and is written:
  \begin{equation}\label{eq:cont.E}
    \mbox{
      $(\uvec{w}_{F_1})_E + (y_{\partial F_1})_{|E} =(\uvec{w}_{F_2})_E + (y_{\partial F_2})_{|E}$
      for all $E\in\ET$ with $\FE=\{F_1,F_2\}$,
    }
  \end{equation}
  where $\FE$ denotes the set containing the two faces of $T$ that share $E$.
  By Proposition \ref{prop:surj.GbdryF}, in order for $y_{\partial F}\in\Poly{k}(\EF)$ to be represented as $\cGpF q_{\partial F}$, it needs to verify 
  \begin{equation}\label{eq:moy.y}
    \sum_{E\in\EF}\omega_{FE}\int_E y_{\partial F} = 0.
  \end{equation}
  We are therefore reduced to finding, for each $F\in\FT$, $y_{\partial F}\in \Poly{k}(\EF)$ such that \eqref{eq:cont.E} and \eqref{eq:moy.y} hold.
  \smallskip
  
  Let us set, for $E\in\ET$ with $\FE=\{F_1,F_2\}$, $r_E\coloneq \frac12\left((y_{\partial F_1})_{|E}+(y_{\partial F_2})_{|E}\right)$.
  We also set, for $F\in\FT$ and $E\in\EF$, $W_{FE}\coloneq\frac12\left( (\uvec{w}_{F'})_{E}-(\uvec{w}_{F})_E\right)$, where $F'$ is the other face of $T$ that shares $E$ with $F$.
  Then, \eqref{eq:cont.E} is equivalent to
  \begin{equation}\label{eq:y.r}
    (y_{\partial F})_{|E}=r_E + W_{FE}\qquad\forall F\in\FT\,,\;\forall E\in\EF.
  \end{equation}
  Using this expression, we obtain the following equivalent reformulation of \eqref{eq:moy.y}:
  \begin{equation}\label{eq:moy.rE}
    \sum_{E\in\EF}\omega_{FE}\int_E r_E = -\sum_{E\in\EF}\omega_{FE}\int_E W_{FE}.
  \end{equation}
  We thus have to find $(r_E)_{E\in\ET}$ such that each $r_E$ belongs to $\Poly{k}(E)$ and \eqref{eq:moy.rE} holds for all $F\in\FT$. Defining then $(y_{\partial F})_{F\in\FT}$ by \eqref{eq:y.r}, we obtain a family that satisfies \eqref{eq:cont.E} and \eqref{eq:moy.y} for all $F\in\FT$.
  The relation \eqref{eq:moy.rE} only involves the integral of $r_E$ over $E$, and we can therefore limit our search to constant polynomials $r_E\in\Real$, in which case \eqref{eq:moy.rE} is recast, after multiplying by $\omega_{TF}$, as
  \begin{equation}\label{eq:moy.rE.2}
    \omega_{TF}\sum_{E\in\EF}\omega_{FE}|E| r_E = -\omega_{TF}\sum_{E\in\EF}\omega_{FE}\int_E W_{FE}\qquad\forall F\in\FT.
  \end{equation}
  This is a linear system of size $\card(\FT)\times \card(\ET)$ in the unknowns $(r_E)_{E\in\ET}\in\Real^{\ET}$. Denoting by $A$ its matrix, this system has a solution if and only if its right-hand side belongs to $\Image A=(\Ker A^t)^\perp$, with $A^t$ denoting the transpose of $A$ through the standard dot products of $\Real^{\FT}$ and $\Real^{\ET}$.
  It is easy to check that $A^t$ corresponds to the mapping
  \[
  \Real^{\FT}\ni(\xi_F)_{F\in\FT}\mapsto \left(\sum_{F\in\FE}\omega_{TF}\omega_{FE}\xi_F\right)_{E\in\ET}\in\Real^{\ET}.
  \]
  Invoking \eqref{eq:orientation} we see that $(\xi_F)_{F\in\FT}\in \Ker A^t$ if and only if $\xi_{F_1}=\xi_{F_2}$ whenever $F_1,F_2\in\FT$ share a common edge.
  Working from neighbouring face to neighbouring face, and using the connectedness of $\partial T$, this shows that the vectors in $\Ker A^t$ are those with all components equal. Hence, the right-hand side of \eqref{eq:moy.rE.2} belongs to $(\Ker A^t)^\perp$ if and only if it is orthogonal to the vector with all components equal to $1$, which translates into
  \[
  0=-\sum_{F\in\FT}\sum_{E\in\EF}\omega_{TF}\omega_{FE}\int_E W_{FE}.
  \]
  Gathering by edges and using \eqref{eq:orientation}, this is equivalent to
  \begin{align}
    0={}& -\sum_{E\in\ET,\,\FE=\{F_1,F_2\}} \omega_{TF_1}\omega_{F_1E}\int_E (W_{F_1E}-W_{F_2E})\nonumber\\
    ={}& -\sum_{E\in\ET,\,\FE=\{F_1,F_2\}} \omega_{TF_1}\omega_{F_1E}\int_E \frac12\left[
      (\uvec{w}_{F_2})_E - (\uvec{w}_{F_1})_E - (\uvec{w}_{F_1})_E + (\uvec{w}_{F_2})_E
      \right]\nonumber\\
    ={}&\sum_{E\in\ET,\,\FE=\{F_1,F_2\}} \left(\omega_{TF_1}\omega_{F_1E}\int_E (\uvec{w}_{F_1})_E
    +\omega_{TF_2}\omega_{F_2E}\int_E (\uvec{w}_{F_2})_E\right)\nonumber\\
    ={}&\sum_{F\in\FT}\omega_{TF}\sum_{E\in\EF}\omega_{FE}\int_E (\uvec{w}_F)_{E},
    \label{eq:solvability}
  \end{align}
  where we have used the definitions of $W_{F_1E}$ and $W_{F_2E}$ in the second line, invoked again \eqref{eq:orientation} in the third line, and concluded gathering back by faces. Using $r_F=1$ in the definition \eqref{eq:cCF} of $\cCF\uvec{w}_F=v_F$, we see that
  \[
  \sum_{E\in\EF}\omega_{FE}\int_E (\uvec{w}_F)_{E}=-\int_F v_F,
  \]
  and the solvability condition \eqref{eq:solvability} of the system \eqref{eq:moy.rE.2} is thus equivalent to \eqref{eq:comp.surj.curl.dT}. Hence, \eqref{eq:moy.rE.2} has at least one solution $(r_E)_{E\in\ET}$ and the proof is complete. \end{proof}


\begin{lemma}[Commutation property for $\cDT$]
  It holds
  \begin{equation}\label{eq:DT.commutation}
    \cDT(\uIdiv\bvec{v}) = \lproj{k}{T}(\DIV\bvec{v})\qquad \forall\bvec{v}\in H^1(T)^3.
  \end{equation}
\end{lemma}

\begin{proof}
  Setting $\uvec{v}_T\coloneq\uIdiv\bvec{v}$, we have $\bvec{v}_T=\Gproj{k-1}{T}\bvec{v} + \GOproj{k}{T}\bvec{v}$ and thus, for all $q_T\in\Poly{k}(T)$, by definition of these projectors and since $\GRAD q_T\in\Goly{k-1}(T)$,
  \[
  \int_T \bvec{v}_T\cdot\GRAD q_T = \int_T \bvec{v}\cdot\GRAD q_T.
  \]
  Since $v_F=\lproj{k}{F}(\bvec{v}\cdot\normal_F)$ and $(q_T)_{|F}\in\Poly{k}(F)$ for all $F\in\FT$, the definition \eqref{eq:DT} of $\cDT\uvec{v}_T$ therefore shows that
  \[
  \int_T \cDT(\uIdiv \bvec{v})~q_T=-\int_T \bvec{v}\cdot\GRAD q_T + \sum_{F\in\FT}\omega_{TF}\int_F (\bvec{v}\cdot\normal_F)q_T=\int_T \DIV \bvec{v}~ q_T,
  \]
  where the conclusion follows from the integration by parts formula \eqref{eq:ipp:div.grad.T}. Since this relation holds for all $q_T\in\Poly{k}(T)$, it proves \eqref{eq:DT.commutation}.
\end{proof}

\subsubsection{Proof of the exactness of the three-dimensional sequence}
\begin{proof}[Proof of Theorem \ref{th:exact.3D}]
  We have to prove that
  \begin{align}\label{eq:Ker.uGT}
    \Ker\uGT &=\uIgrad\Real,
    \\\label{eq:Ker.uCT}
    \Ker\uCT &= \Image\uGT,
    \\\label{eq:Ker.DT}
    \Ker\cDT &= \Image\uCT,
    \\\label{eq:Image.DT}
    \Image\cDT &= \Poly{k}(T).
  \end{align}
  \underline{1. \emph{Proof of \eqref{eq:Ker.uGT}.}}
  By the consistency properties of the boundary and volume gradients (see \eqref{eq:GE.k.cons}, \eqref{eq:GF.k.cons} and \eqref{eq:GT:polynomial.consistency}), it holds that $\cGE(\uIgrad[\partial T]C)=0$, $\cGF(\uIgrad[\partial T]C)=\bvec{0}$, and $\cGT(\uIgrad C) = \bvec{0}$ for all $C\in\Real$, proving by definition \eqref{eq:def.uGT} of $\uGT$ that
  \[
  \boxed{\uIgrad\Real \subset \Ker\uGT.}
  \]
  
  Let us prove the converse inclusion, i.e.,
  \begin{equation}\label{eq:Ker.uGT.subset.uIgrad.Real}
    \boxed{\Ker\uGT\subset\uIgrad\Real.}
  \end{equation}
  Let $\underline{q}_T\in\uHgrad$ be such that $\uGT\underline{q}_T=\underline{\bvec{0}}$.
  Then, $\uGF\underline{q}_{\partial T}=\uvec{0}$ for all $F\in\FT$ and thus, recalling the two-dimensional exactness proved in Theorem \ref{th:exact.2D} and accounting for the single-valuedness of vertex and edge unknowns, there exists $C\in\Real$ such that $\underline{q}_{\partial T}=\uIgrad[\partial T] C$.
  Therefore, it only remains to prove that $q_T=C$.
  Enforcing $\GOT\underline{q}_T=\bvec{0}$ and using $\GOT=\ROproj{k}{T}\cGT$ (cf. \eqref{eq:def.uGT}) together with the definition \eqref{eq:GT.k} of $\cGT$ and the fact that $\trF\underline{q}_{\partial T}=\trF(\uIgrad[\partial T]C) = C$ by the polynomial consistency \eqref{eq:trF:polynomial.consistency} of this trace reconstruction operator, we infer, for all $\bvec{w}_T\in\Roly{k}(T)^\perp$,
  \[
  0 = \int_T \GOT\underline{q}_T\cdot\bvec{w}_T=\int_T \cGT\underline{q}_T\cdot\bvec{w}_T=-\int_T q_T\DIV\bvec{w}_T + \sum_{F\in\FT}\omega_{TF}\int_F C (\bvec{w}_T\cdot\normal_F)
  = \int_F(q_T-C)\DIV\bvec{w}_T,
  \]
  where we have used the integration by parts formula \eqref{eq:ipp:div.grad.T} with $C$ instead of $q_T$ to conclude.
  Since $\DIV:\Roly{k}(T)^\perp\to\Poly{k-1}(T)$ is surjective by \eqref{eq:iso:DIV} and $q_T\in\Poly{k-1}(T)$, this implies $q_T=\lproj{k-1}{T}C$, thus proving \eqref{eq:Ker.uGT.subset.uIgrad.Real}. 
  \medskip\\
  \underline{2. \emph{Proof of \eqref{eq:Ker.uCT}.}} We start by proving that
  \begin{equation}\label{eq:uGT.uHgrad.subset.Ker.uCT}
    \boxed{\Image\uGT\subset\Ker\uCT,}
  \end{equation}
  that is $\uCT(\uGT\underline{q}_T)=\underline{\bvec{0}}$ for all $\underline{q}_T\in\uHgrad$.
  Theorem \ref{th:exact.2D} implies $\cCF(\uGT\underline{q}_T)=0$ for all $F\in\FT$.
  Let us prove that $\CGT(\uGT\underline{q}_T)=\bvec{0}$.
  From the characterisation \eqref{eq:rep.CGT} of $\CGT$ we infer, for all $\bvec{w}_T\in\Goly{k-1}(T)$,
  \[
  \int_T\CGT(\uGT\underline{q}_T)\cdot\bvec{w}_T
  =
  \sum_{F\in\FT}\omega_{TF}\int_F(\GRF\underline{q}_T+\GROF\underline{q}_T)\cdot(\bvec{w}_T\times\normal_F).
  \]
  Since $\bvec{w}_T\in\GRAD\Poly{k}(T)$, the relation \eqref{eq:gradT.rotF} shows that $\bvec{w}_{T|F}\times\normal_F\in \Roly{k-1}(F)\subset\Roly{k}(F)$ and thus, using $\GROF\underline{q}_T\in\Roly{k}(F)^\perp$ and the relation \eqref{eq:GF} with $\bvec{w}_F=\bvec{w}_{T|F}\times\normal_F$, we continue with
  \begin{align*}
    \int_T\CGT(\uGT\underline{q}_T)\cdot\bvec{w}_T={}&
    \sum_{F\in\FT}\omega_{TF}\int_F\GRF\underline{q}_T\cdot(\bvec{w}_T\times\normal_F)
    \\
    ={}& \sum_{F\in\FT}\omega_{TF} \sum_{E\in\EF}\omega_{FE}\int_E q_E (\bvec{w}_T\times\normal_F)\cdot\normal_{FE}
    \\
    ={}& \sum_{F\in\FT}\omega_{TF} \sum_{E\in\EF}\omega_{FE}\int_E q_E (\bvec{w}_T\cdot\tangent_E)
    = 0,
  \end{align*}
  where we have used \eqref{eq:righthanded} to pass to the third line and \eqref{eq:cancel.edges} with $a_E=\int_Eq_E(\bvec{w}_T\cdot\tangent_E)$ to conclude.
  This implies $\CGT(\uGT\underline{q}_T)=\bvec{0}$.
  \smallskip\\
  We next notice that it holds, for all $\bvec{w}_T\in\Goly{k}(T)^\perp$,
  \[
  \begin{aligned}
    \int_T\CGOT(\uGT\underline{q}_T)\cdot\bvec{w}_T
    &=
    \int_T(\GRT\underline{q}_T+\GROT\underline{q}_T)\cdot\CURL\bvec{w}_T
    + \sum_{F\in\FT}\omega_{TF}\int_F\trFt(\uGF\underline{q}_{\partial T})\cdot(\bvec{w}_T\times\normal_F)
    \\
    &=
    \int_T\GRT\underline{q}_T\cdot\CURL\bvec{w}_T
    + \sum_{F\in\FT}\omega_{TF}\int_F\cGF\underline{q}_{\partial T}\cdot(\bvec{w}_T\times\normal_F)
    =0,
  \end{aligned}
  \]
  where we have used the definition \eqref{eq:CT.k} of $\cCT$ together with $\CGOT=\GOproj{k}{T}\cCT$ (cf. \eqref{eq:def.uCT}) in the first line, the relation $\GROT=\ROproj{k}{T}\cGT$ (cf. \eqref{eq:def.uGT}) together with $\CURL\bvec{w}_T\in \Roly{k-1}(T)$ and the property \eqref{eq:trFt:GF} of the tangential trace reconstruction in the second line, and the link \eqref{eq:link.GT.GF} between volume and face gradients to conclude. This proves \eqref{eq:uGT.uHgrad.subset.Ker.uCT}.
  \medskip

  We next prove the converse inclusion, that is,
  \begin{equation}\label{eq:Ker.uCT.subset.uGT.uHgrad}
    \boxed{\Ker\uCT\subset\Image\uGT.}
  \end{equation}
  This requires to show that, for all $\uvec{v}_T\in\uHcurl$ such that $\uCT\uvec{v}_T=\underline{\bvec{0}}$, there exists $\underline{q}_T\in\uHgrad$ such that $\uvec{v}_T=\uGT\underline{q}_T$.
  For all $F\in\FT$, enforcing $\cCF\uvec{v}_T=0$ in \eqref{eq:cCF} and taking $r_F=1$, we see that $\sum_{E\in\EF}\omega_{FE}\int_E v_E=0$.
  Hence, Proposition \ref{prop:surj.GbdryF} provides $q_{\partial F}\in\Poly[\rm c]{k+1}(\partial F)$ such that $v_E=\cGE q_{\partial F}$ for all $E\in\EF$.
  Each function $q_{\partial F}$ for $F\in\FT$ is defined up to an additive constant which, by the single-valuedness of $(v_E)_{E\in\ET}$ across the faces, can be selected so as to form a continuous function $q_{\partial^2 T}\in\Poly[\rm c]{k+1}(\partial^2 T)$ defined on the whole $\partial^2 T$, and such that $v_E=\cGE q_{\partial^2 T}$ for all $E\in\ET$.
  \smallskip
  
  We next proceed as in Point 3 of the proof of Proposition \ref{prop:almost.exact}:
  first, to infer that, for any choice of $q_{\partial T}=(q_F)_{F\in\FT}\in\Poly{k-1}(\FT)$, letting $\underline{q}_{\partial T}\coloneq( q_{\partial T}, q_{\partial^2 T})$, it holds $\bvec{v}_{\cvec{R},F}=\GRF \underline{q}_{\partial T}$;
  then, to select a proper $q_{\partial T}$ such that $\bvec{v}_{\cvec{R},F}^\perp=\GROF\underline{q}_{\partial T}$ for all $F\in\FT$.
  This proves that $\uvec{v}_{F}=\uGF\underline{q}_{\partial T}$ for all $F\in\FT$.
  
  Let us now show that, for any $q_T\in \Poly{k-1}(T)$, setting $\underline{q}_T\coloneq(q_T,\underline{q}_{\partial T})\in\uHgrad$, it holds $\bvec{v}_{\cvec{R},T}=\GRT\underline{q}_{T}$.
  Applying the definition \eqref{eq:CT.k} of $\cCT$ to an arbitrary $\bvec{w}_T\in\Goly{k}(T)^\perp$ and enforcing $\bvec{0}=\CGOT\uvec{v}_T=\GOproj{k}{T}(\cCT\underline{q}_T)$, we get
  \[
  \begin{aligned}
    \int_T\bvec{v}_T\cdot\CURL\bvec{w}_T
    &= -\sum_{F\in\FT}\omega_{TF}\int_F\trFt\uvec{v}_{\partial T}\cdot(\bvec{w}_T\times\normal_F)
    \\
    &= -\sum_{F\in\FT}\omega_{TF}\int_F\trFt(\uGF\underline{q}_{\partial T})\cdot(\bvec{w}_T\times\normal_F)
    \\
    &= -\sum_{F\in\FT}\omega_{TF}\int_F \cGF\underline{q}_{\partial T}\cdot(\bvec{w}_T\times\normal_F)
    \\
    &= \int_T \GRT\underline{q}_T\cdot\CURL\bvec{w}_T,
  \end{aligned}
  \]
  where we have used the property \eqref{eq:trFt:GF} of the tangential trace reconstruction to pass to the third line and the link \eqref{eq:link.GT.GF} between face and volume gradients to conclude.
  By \eqref{eq:iso:CURL}, $\CURL\bvec{w}_T$ spans $\Roly{k-1}(T)$ when $\bvec{w}_T$ spans $\Goly{k}(T)^\perp$, and we therefore deduce that $\GRT\underline{q}_T=\Rproj{k-1}{T}\bvec{v}_T=\bvec{v}_{\cvec{R},T}$ as desired.
  \smallskip

  It only remains to prove the existence of $q_T\in\Poly{k-1}(T)$ such that $\bvec{v}_{\cvec{R},T}^\perp=\GROT\underline{q}_T$. Recalling that $\GROT\underline{q}_T=\ROproj{k}{T}(\cGT\underline{q}_T)$ (see \eqref{eq:def.uGT}) and applying the definition \eqref{eq:GT.k} of $\cGT\underline{q}_T$ to an arbitrary test function $\bvec{w}_T\in\Roly{k}(T)^\perp$, this requires the following condition to hold:
  \[
  \int_T q_T\DIV\bvec{w}_T
  = -\int_T\bvec{v}_{\cvec{R},T}^\perp\cdot\bvec{w}_T
  + \sum_{F\in\FT}\omega_{TF}\int_F\trF\underline{q}_{\partial T}(\bvec{w}_T\cdot\normal_F)  \qquad\forall\bvec{w}_T\in\Roly{k}(T)^\perp,
  \]
  which appropriately defines $q_T$ since $\DIV:\Roly{k}(T)^\perp\to\Poly{k-1}(T)$ is an isomorphism by \eqref{eq:iso:DIV}.
  This concludes the proof of \eqref{eq:Ker.uCT.subset.uGT.uHgrad}.
  \medskip\\
  \underline{3. \emph{Proof of \eqref{eq:Ker.DT}.}}
  Let us start by proving that $\cDT(\uCT\uvec{v}_T)=0$ for all $\uvec{v}_T\in\uHcurl$, which implies
  \[
  \boxed{\Image\uCT\subset\Ker(\cDT).}
  \]
  For all $q_T\in\Poly{k}(T)$, we have
  \[
  \int_T\cDT(\uCT\uvec{v}_T) q_T
  = -\int_T(\CGT\uvec{v}_T+\cancel{\CGOT\uvec{v}_T})\cdot\GRAD q_T + \sum_{F\in\FT}\omega_{TF}\int_F\cCF\uvec{v}_T~q_T
  =0,
  \]
  where we have used the definition \eqref{eq:DT} of $\cDT$ in the first equality, the $L^2$-orthogonality of $\CGOT\underline{q}_T$ and $\GRAD q_T\in\Goly{k-1}(T)$ in the cancellation, and we have concluded using the link \eqref{eq:CT.CF} between volume and face curls. Since $q_T$ is arbitrary in $\Poly{k}(T)$, this shows that $\cDT(\uCT\uvec{v}_T)=0$.
  \smallskip

  Let us now prove the inclusion
  \begin{equation}\label{eq:Ker.DT.subset.uCT.uHcurl}
    \boxed{\Ker(\cDT)\subset\Image\uCT.}
  \end{equation}
  We fix an element $\uvec{v}_T\in\uHdiv$ such that $\cDT\uvec{v}_T=0$ and prove the existence of $\uvec{z}_T\in\uHcurl$ such that $\uvec{v}_T=\uCT\uvec{z}_T$.
  Enforcing $\cDT\uvec{v}_T=0$ in \eqref{eq:DT} with $q_T=1$, we infer that $\sum_{F\in\FT}\omega_{TF}\int_Fv_F=0$. Lemma \ref{lem:surj.CurlF} then provides $\uvec{z}_{\partial T}\in\uHcurl[\partial T]$ such that $v_F=\cCF\uvec{z}_{\partial T}$ for all $F\in\FT$.
  \smallskip

  Enforcing again $\cDT\uvec{v}_T=0$ in \eqref{eq:DT}, this time for a generic test function $q_T\in\Poly{k}(T)$, and accounting for the previous result, we can write, for all $\uvec{z}_T\in\uHcurl$ whose boundary values are given by $\uvec{z}_{\partial T}$,
  \[
  \int_T\bvec{v}_T\cdot\GRAD q_T
  = \sum_{F\in\FT}\omega_{TF}\int_F\cCF\uvec{z}_{\partial T}~q_T
  = \int_T\CGT\uvec{z}_T\cdot\GRAD q_T,
  \]
  where the conclusion follows from the relation \eqref{eq:CT.CF} linking volume and face curls.
  Since $\GRAD q_T$ spans $\Goly{k-1}(T)$ as $q_T$ spans $\Poly{k}(T)$, this proves that $\CGT\uvec{z}_T=\Gproj{k-1}{T}\bvec{v}_T=\bvec{v}_{\cvec{G},T}$.
  \smallskip

  Finally, we show that, for some $\bvec{z}_T\in\Roly{k-1}(T)$, the vector $\uvec{z}_T=(\bvec{z}_T,\uvec{z}_{\partial T})\in\uHcurl$ satisfies $\bvec{v}_{\cvec{G},T}^\perp=\CGOT\uvec{z}_T$.
  Recalling the definition \eqref{eq:CT.k} of the full curl reconstruction and \eqref{eq:def.uCT} to write $\CGOT=\GOproj{k}{T}\cCT$, this amounts to enforcing the following condition: For all $\bvec{w}_T\in\Goly{k}(T)^\perp$,
  \[
  \int_T\bvec{z}_T\cdot\CURL\bvec{w}_T
  = \int_T\bvec{v}_{\cvec{G},T}^\perp\cdot\bvec{w}_T
  - \sum_{F\in\FT}\omega_{TF}\int_F\trFt\uvec{z}_{\partial T}\cdot(\bvec{w}_T\times\normal_F).
  \]
  By \eqref{eq:iso:CURL}, $\CURL:\Goly{k}(T)^\perp\to\Roly{k-1}(T)$ is an isomorphism and this relation therefore defines a unique $\bvec{z}_T\in\Roly{k-1}(T)$.
  This concludes the proof of \eqref{eq:Ker.DT.subset.uCT.uHcurl}.
  \medskip\\
  \underline{4. \emph{Proof of \eqref{eq:Image.DT}.}}
  Let $q_T\in\Poly{k}(T)$ and let us show the existence of $\uvec{v}_T\in\uHdiv$ such that $q_T=\cDT\uvec{v}_T$.
  By \eqref{eq:iso:DIV}, there exists $\bvec{v}\in\Roly{k+1}(T)^\perp$ such that $\DIV\bvec{v}=q_T$.
  Using the polynomial consistency of $\lproj{k}{T}$ followed by the commutation property \eqref{eq:DT.commutation}, we have $q_T=\DIV\bvec{v}=\lproj{k}{T}(\DIV\bvec{v})=\cDT(\uIdiv\bvec{v})$, which is the desired result with $\uvec{v}_T=\uIdiv\bvec{v}$.
\end{proof}

\subsection{Commutative diagrams}

In this section we prove commutative diagram properties for the discrete three-dimensional sequence.
These commutative diagrams express, in a synthetic manner, crucial compatibility properties of the discrete three-dimensional sequence \eqref{eq:sequence.3D}.
To this end, we recall the definition \eqref{def:nedelec} of the N\'ed\'elec space and we introduce the Raviart--Thomas--N\'ed\'elec space
\begin{equation}\label{eq:RT}
  \RT{k}(T) \coloneq \Poly{k}(T)^3 + (\bvec{x}-\overline{\bvec{x}}_T)\Poly{k}(T).
\end{equation}

\begin{theorem}[Commutative diagrams]\label{thm:commutative.diagrams}
  Denoting by $i_T:\Poly{k}(T)\to\Poly{k}(T)$ the identity operator, the following diagrams commute:
  \begin{equation}\label{eq:commutative.diagrams}
    \begin{tikzcd}
      \Poly{k+1}(T)\arrow{r}{\GRAD}\arrow{d}{\uIgrad}
      & \NE{k}(T)\arrow{r}{\CURL}\arrow{d}{\uIcurl}
      & \RT{k}(T)\arrow{r}{\DIV}\arrow{d}{\uIdiv}
      & \Poly{k}(T)\arrow{d}{i_T}
      \\
      \uHgrad\arrow{r}{\uGT} & \uHcurl\arrow{r}{\uCT} & \uHdiv\arrow{r}{\cDT} & \Poly{k}(T)
    \end{tikzcd}
  \end{equation}
\end{theorem}
\begin{proof}
  Using the polynomial consistency properties \eqref{eq:GT:polynomial.consistency} of $\cGT$, \eqref{eq:GF.k.cons} of $\cGF$, and \eqref{eq:GE.k.cons} of $\cGE$ we infer, for all $q\in\Poly{k+1}(T)$,
  \[
  \begin{gathered}
    \cGT(\uIgrad q) = \GRAD q,
    \qquad
    \cGF(\uIgrad q) = \GRAD_F q_{|F} \quad \forall F\in\FT,
    \\
    \cGE(\uIgrad q) = (q_{|E})' \quad \forall E\in\ET.
  \end{gathered}
  \]
  Plugging these relations into the definition \eqref{eq:def.uGT} of $\uGT$ and recalling the definition \eqref{eq:uIcurlT} of $\uIcurl$ proves the leftmost commutative diagram in \eqref{eq:commutative.diagrams}.
  \medskip\\
  The commutation properties \eqref{eq:CT.k:polynomial.consistency} of $\cCT$ and \eqref{eq:CF.commutation} of $\cCF$, together with $(\uIcurl\bvec{v})_{F}=\uIrot(\normal_F\times(\bvec{v}_{|F}\times\normal_F))$ give, for all $\bvec{v}\in\NE{k}(T)$,
  \[
  \cCT(\uIcurl\bvec{v}) = \CURL\bvec{v},\qquad
  \cCF(\uIcurl\bvec{v}) = \ROT_F(\normal_F\times(\bvec{v}_{|F}\times\normal_F)) = (\CURL\bvec{v})_{|F}\cdot\normal_F\quad\forall F\in\FT,
  \]
  where we have additionally used the fact that $\ROT_F(\normal_F\times(\bvec{v}_{|F}\times\normal_F))\in\Poly{k}(F)$ and the identity \eqref{eq:curlT.divF.rotF}.
  Plugging these relations into the definition \eqref{eq:def.uCT} of $\uCT$ and recalling the definition \eqref{eq:uIdivT} of $\uIdiv$ concludes the proof of the middle commutative diagram in \eqref{eq:commutative.diagrams}.
  \medskip\\
  Finally, the rightmost commutative diagram follows combining \eqref{eq:DT.commutation} and $i_T=\lproj{k}{T}$ on $\Poly{k}(T)$.
\end{proof}

\subsection{Three-dimensional potentials}

\subsubsection{Scalar potential}

Starting from the full gradient \eqref{eq:GT.k} and scalar trace reconstructions $(\trF)_{F\in\FT}$ satisfying the properties \eqref{eq:trF:prop}, we define a scalar potential reconstruction $\pgrad:\uHgrad\to\Poly{k+1}(T)$ as follows:
For all $\underline{q}_T\in\uHgrad$,
\begin{equation}\label{eq:pgradT}
  \int_T\pgrad\underline{q}_T\DIV\bvec{v}_T
  = -\int_T\cGT\underline{q}_T\cdot\bvec{v}_T
  + \sum_{F\in\FT}\omega_{TF}\int_F\trF\underline{q}_{\partial T}(\bvec{v}_T\cdot\normal_F)
  \qquad\forall\bvec{v}_T\in\Roly{k+2}(T)^\perp.
\end{equation}
This relation defines a unique $\pgrad\underline{q}_T$ since $\DIV:\Roly{k+2}(T)^\perp\to\Poly{k+1}(T)$ is an isomorphism by \eqref{eq:iso:DIV}.
Combining the polynomial consistencies \eqref{eq:GT:polynomial.consistency} of the full gradient and \eqref{eq:trF:polynomial.consistency} of the scalar trace reconstructions with the integration by parts formula \eqref{eq:ipp:div.grad.T}, it is inferred that
\begin{equation}\label{eq:pgradT:polynomial.consistency}
  \pgrad(\uIgrad q) = q\qquad\forall q\in\Poly{k+1}(T).
\end{equation}
Notice that other choices are possible for a scalar potential reconstruction satisfying \eqref{eq:pgradT:polynomial.consistency}.
  In the spirit of Remark \ref{rem:trFt:alternative} one can take, e.g., $\pgrad\underline{q}_T-\lproj{k-1}{T}(\pgrad\underline{q}_T)+q_T$ as three-dimensional scalar potential, which has the additional property that its $L^2$-orthogonal projection on $\Poly{k-1}(T)$ coincides with the internal unknown $q_T$.

\subsubsection[Vector potential on XcurlT]{Vector potential on $\uHcurl$}

A vector potential reconstruction $\pcurl:\uHcurl\to\Poly{k}(T)^3$ is obtained as follows:
For all $\uvec{v}_T\in\uHcurl$,
\begin{subequations}\label{eq:pcurlT}
  \begin{alignat}{2} \label{eq:pcurlT:RTk}
    \int_T\pcurl\uvec{v}_T\cdot\CURL\bvec{w}_T
    &=
    \int_T\cCT\uvec{v}_T{\cdot}\bvec{w}_T
    -\hspace{-1ex}
    \sum_{F\in\FT}\!\!\omega_{TF}\!\!\int_F\trFt\uvec{v}_{\partial T}{\cdot}\,(\bvec{w}_T\times\normal_F)
    &\quad&
    \forall\bvec{w}_T\in\Goly{k+1}(T)^\perp,
    \\ \label{eq:pcurlT:RTk.perp}
    \int_T\pcurl\uvec{v}_T\cdot\bvec{w}_T
    &= \int_T\bvec{v}_{\cvec{R},T}^\perp\cdot\bvec{w}_T
    &\quad&\forall\bvec{w}_T\in\Roly{k}(T)^\perp.
  \end{alignat}
\end{subequations}
To check that these equations define a unique $\pcurl\uvec{v}_T\in\Poly{k}(T)^3$, observe that  \eqref{eq:pcurlT:RTk} and \eqref{eq:pcurlT:RTk.perp} prescribe, respectively, $\Rproj{k}{T}(\pcurl\uvec{v}_T)$ (since $\CURL:\Goly{k+1}(T)^\perp\to\Roly{k}(T)$ is an isomorphism, see \eqref{eq:iso:CURL}) and $\ROproj{k}{T}(\pcurl\uvec{v}_T)$, and recall the orthogonal decomposition $\Poly{k}(T)^3=\Roly{k}(T)\oplus\Roly{k}(T)^\perp$.
In the spirit of Remark \ref{rem:trFt:alternative}, an alternative vector potential reconstruction in $\uHcurl$ is $\pcurl\uvec{v}_T-\Rproj{k-1}{T}(\pcurl\uvec{v}_T)+\bvec{v}_{\cvec{R},T}$, which has the additional property that its $L^2$-orthogonal projection on $\Roly{k-1}(T)$ coincides with the internal unknown $\bvec{v}_{\cvec{R},T}$.
\begin{proposition}[Consistency of $\pcurl$]
  It holds
  \begin{equation}\label{eq:pcurlT:polynomial.consistency}
    \pcurl(\uIcurl\bvec{v})=\bvec{v} \qquad\forall\bvec{v}\in\Poly{k}(T)^3.
  \end{equation}
\end{proposition}
\begin{proof}
  Let $\bvec{v}\in\Poly{k}(T)^3$. Applying \eqref{eq:pcurlT:RTk} to $\uvec{v}_T=\uIcurl\bvec{v}$ and using the consistency properties \eqref{eq:CT.k:polynomial.consistency} of $\cCT$ and \eqref{eq:trFt:polynomial.consistency} of $\trFt$ we obtain, for all $\bvec{w}_T\in\Goly{k+1}(T)^\perp$,
  \[
  \int_T\pcurl(\uIcurl\bvec{v})\cdot\CURL\bvec{w}_T
  =
  \int_T\CURL\bvec{v}\cdot\bvec{w}_T-
  \sum_{F\in\FT}\omega_{TF}\int_F\vlproj{k}{F}(\normal_F\times(\bvec{v}_{|F}\times\normal_F))\cdot(\bvec{w}_T\times\normal_F).
  \]
  Since $\normal_F\times(\bvec{v}_{|F}\times\normal_F)\in\Poly{k}(F)^2$, the projector $\vlproj{k}{F}$ above can be removed and, invoking the integration by parts formula \eqref{eq:ipp:curl.T}, it is inferred that $\Rproj{k}{T}(\pcurl(\uIcurl\bvec{v}))=\Rproj{k}{T}\bvec{v}$.
  On the other hand, \eqref{eq:pcurlT:RTk.perp} and the definition \eqref{eq:uIcurlT} of $\uIcurl$ readily imply $\ROproj{k}{T}(\pcurl(\uIcurl\bvec{v}))=\ROproj{k}{T}\bvec{v}$.
  Recalling the orthogonal decomposition $\Poly{k}(T)^3=\Roly{k}(T)\oplus\Roly{k}(T)^\perp$, the conclusion follows.
\end{proof}

\subsubsection[Vector potential on XdivT]{Vector potential on $\uHdiv$}

The vector potential in $\uHdiv$ is $\pdiv:\uHdiv\to\Poly{k}(T)^3$ such that, for all $\uvec{v}_T\in\uHdiv$,
\begin{subequations}\label{eq:pdivT}
  \begin{alignat}{2} \label{eq:pdivT:Goly.k}
    \int_T\pdiv\uvec{v}_T\cdot\GRAD q_T
    &= -\int_T\cDT\uvec{v}_T q_T
    + \sum_{F\in\FT}\omega_{TF}\int_F v_F q_T
    &\qquad&\forall q_T\in\Poly{0,k+1}(T),
    \\ \label{eq:pdivT:Goly.k.perp}
    \int_T\pdiv\uvec{v}_T\cdot\bvec{w}_T
    &=\int_T\bvec{v}_{\cvec{G},T}^\perp\cdot\bvec{w}_T
    &\qquad&\forall\bvec{w}_T\in\Goly{k}(T)^\perp.
  \end{alignat}
\end{subequations}
These equations prescribe, respectively, $\Gproj{k}{T}(\pdiv\uvec{v}_T)$ and $\GOproj{k}{T}(\pdiv\uvec{v}_T)$, hence $\pdiv\uvec{v}_T$ by virtue of the orthogonal decomposition $\Poly{k}(T)^3=\Goly{k}(T)\oplus\Goly{k}(T)^\perp$.
In the spirit of Remark \ref{rem:trFt:alternative}, an alternative vector potential reconstruction in $\uHdiv$ is $\pdiv\uvec{v}_T-\Gproj{k-1}{T}(\pdiv\uvec{v}_T)+\bvec{v}_{\cvec{G},T}$, which has the additional property that its $L^2$-orthogonal projection on $\Goly{k-1}(T)$ coincides with the internal unknown $\bvec{v}_{\cvec{G},T}$.
\begin{proposition}[Consistency of $\pdiv$]
  It holds:
  \begin{equation}\label{eq:pdivT:polynomial.consistency}
    \pdiv(\uIdiv\bvec{v}) = \vlproj{k}{T}\bvec{v} \qquad \forall\bvec{v}\in\RT{k}(T).
  \end{equation}
\end{proposition}
\begin{proof}
  Let $\bvec{v}\in\RT{k}(T)$ and set $\uvec{v}_T=\uIdiv\bvec{v}$.
  Recalling the commutation property \eqref{eq:DT.commutation} of $\cDT$ and the definition \eqref{eq:uIdivT} of $\uIdiv$, \eqref{eq:pdivT:Goly.k} gives:
  For all $q_T\in\Poly{0,k+1}(T)$,
  \[
  \int_T\pdiv(\uIdiv\bvec{v})\cdot\GRAD q_T
  = -\int_T\lproj{k}{T}(\DIV\bvec{v}) q_T
  + \sum_{F\in\FT}\omega_{TF}\int_F \lproj{k}{F}(\bvec{v}\cdot\normal_F) q_T.
  \]
  Since $\bvec{v}\in\RT{k}(T)$, we have $\DIV\bvec{v}\in\Poly{k}(T)$ and $\bvec{v}_{|F}\cdot\normal_F\in\Poly{k}(F)$ for all $F\in\FT$ (this is a consequence of the definition \eqref{eq:RT} of the Raviart--Thomas space observing that, for all $F\in\FT$, the mapping $F\ni\bvec{x}\mapsto(\bvec{x}-\overline{\bvec{x}}_T)\cdot\normal_F\in\Real$ is constant); hence, the projectors can be removed from the right-hand side of the above expression.
  Invoking then the integration by parts formula \eqref{eq:ipp:div.grad.T}, it is inferred that $\Gproj{k}{T}(\pdiv(\uIdiv\bvec{v}))=\Gproj{k}{T}\bvec{v}$.
  Equation \eqref{eq:pdivT:Goly.k.perp}, on the other hand, readily implies $\GOproj{k}{T}(\pdiv(\uIdiv\bvec{v}))=\GOproj{k}{T}\bvec{v}$.
  Combining these facts with the orthogonal decomposition $\Poly{k}(T)^3=\Goly{k}(T)\oplus\Goly{k}(T)^\perp$, \eqref{eq:pdivT:polynomial.consistency} follows.
\end{proof}

\subsection{Three-dimensional discrete $L^2$-products}\label{sec:L2prod.3D}

We next define discrete counterparts of the $L^2$-products in $H^1(T)$, $\Hcurl{T}$, and $\Hdiv{T}$:
\begin{itemize}
\item $(\cdot,\cdot)_{\GRAD,T}:\uHgrad\times\uHgrad\to\Real$ such that, for all $\underline{q}_T,\underline{r}_T\in\uHgrad$,
  \begin{equation}\label{eq:().grad.T}
    \begin{aligned}
      (\underline{q}_T,\underline{r}_T)_{\GRAD,T}
      \coloneq{}&\int_T\pgrad\underline{q}_T~\pgrad\underline{r}_T + \int_T \delta_{\GRAD,T}^{k-1}\underline{q}_T~\delta_{\GRAD,T}^{k-1}\underline{r}_T\\
      &+ \sum_{F\in\FT} h_T\int_F\delta_{\GRAD,F}^{k-1}\underline{q}_T~\delta_{\GRAD,F}^{k-1}\underline{r}_T
      +\sum_{E\in\ET} h_T^2\int_E\delta_{\GRAD,E}^{k+1}\underline{q}_T~\delta_{\GRAD,E}^{k+1}\underline{r}_T,
    \end{aligned}
  \end{equation}
  where $h_T$ is the diameter of $T$ and we have set, for any $\underline{q}_T\in\uHgrad$,
  \[
  \big(
  \delta_{\GRAD,T}^{k-1}\underline{q}_T,
  (\delta_{\GRAD,F}^{k-1}\underline{q}_T)_{F\in\FT},
  (\delta_{\GRAD,E}^{k+1}\underline{q}_T)_{E\in\ET}
  \big)
  \coloneq
  \uIgrad(\pgrad\underline{q}_T)-\underline{q}_{T}.
  \]
\item $(\cdot,\cdot)_{\CURL,T}:\uHcurl\times\uHcurl\to\Real$ such that, for all $\uvec{v}_T,\uvec{w}
  _T\in\uHcurl$, 
  \begin{equation}\label{eq:().curl.T}
    \begin{aligned}
      (\uvec{v}_T,\uvec{w}_T)_{\CURL,T}
      \coloneq{}&\int_T\pcurl\uvec{v}_T\cdot\pcurl\uvec{w}_T\\
      &+\int_T\bvec{\delta}_{\CURL,T}^{k-1}\uvec{v}_T\cdot\bvec{\delta}_{\CURL,T}^{k-1}\uvec{w}_T      
        \\
      &+ \sum_{F\in\FT}h_T\int_F\left(
      \bvec{\delta}_{\CURL,F}^{k-1}\uvec{v}_T\cdot\bvec{\delta}_{\CURL,F}^{k-1}\uvec{w}_T  
      + \bvec{\delta}_{\CURL,F}^{\perp,k}\uvec{v}_T\cdot\bvec{\delta}_{\CURL,F}^{\perp,k}\uvec{w}_T
      \right)
      \\
      &
      + \sum_{E\in\ET}h_T^2\int_E\delta_{\CURL,E}^k\uvec{v}_T~\delta_{\CURL,E}^k\uvec{w}_T,
    \end{aligned}
  \end{equation}
  where we have set, for all $\uvec{v}_T\in\uHcurl$, with obvious notations,
  \[
  \big(
  \bvec{\delta}_{\CURL,T}^{k-1}\uvec{v}_T+\bvec{\delta}_{\CURL,T}^{\perp,k}\uvec{v}_T,
  (\bvec{\delta}_{\CURL,F}^{k-1}\uvec{v}_T+\bvec{\delta}_{\CURL,F}^{\perp,k}\uvec{v}_T)_{F\in\FT},
  (\delta_{\CURL,E}^k\uvec{v}_T)_{E\in\ET}
  \big)
  \coloneq\uIcurl(\pcurl\uvec{v}_T) - \uvec{v}_{T}.
  \]
  Notice that,  by \eqref{eq:pcurlT:RTk.perp}, it holds $\bvec{\delta}_{\CURL,T}^{\perp,k}\uvec{v}_T=\bvec{0}$, so there is no need to penalise this quantity in \eqref{eq:().curl.T}.
\item $(\cdot,\cdot)_{\DIV,T}:\uHdiv\times\uHdiv\to\Real$ such that, for all $\uvec{v}_T,\uvec{w}_T\in\uHdiv$,
  \begin{equation}\label{eq:().div.T}
    \begin{aligned}
      (\uvec{v}_T,\uvec{w}_T)_{\DIV,T}
      &\coloneq
      \int_T\pdiv\uvec{v}_T\cdot\pdiv\uvec{w}_T
      + \sum_{F\in\FT}h_T\int_F\delta_{\DIV,F}^k\uvec{v}_T~\delta_{\DIV,F}^k\uvec{w}_T,
    \end{aligned}
  \end{equation}
  where we have set, for all $\uvec{v}_T\in\uHdiv$,
  \[
  \delta_{\DIV,F}^k\uvec{v}_T
  \coloneq \pdiv\uvec{v}_T\cdot\normal_F-v_F
  \qquad\forall F\in\FT.
  \]
\end{itemize}

\begin{remark}[Discrete $L^2$-product in $\uHdiv$]
    Also for $\uHdiv$ it is possible to define a discrete $L^2$-product where all the components of $\uIdiv(\pdiv\uvec{v}_T)-\uvec{v}_T$ are penalised.
    It turns out, however, that penalising the volume differences is not required to prove definiteness; cf. the proof of Lemma \ref{lem:L2.products:3d} below.
\end{remark}

\begin{lemma}[Discrete $L^2$-products]\label{lem:L2.products:3d}
  The bilinear forms $(\cdot,\cdot)_{\bullet,T}$, with $\bullet\in\{\GRAD,\CURL,\DIV\}$, are positive definite.
  Additionally, they satisfy the following consistency properties:
  \begin{alignat}{2}
    \label{eq:().grad.cons}
    (\uIgrad q,\uIgrad r)_{\GRAD,T}&=(q,r)_{L^2(T)}&\qquad&\forall q,r\in\Poly{k+1}(T),\\
    \label{eq:().curl.cons}
    (\uIcurl\bvec{v},\uIcurl \bvec{w})_{\CURL,T}&=(\bvec{v},\bvec{w})_{L^2(T)^3}&\qquad&\forall \bvec{v},\bvec{w}\in\Poly{k}(T)^3,\\
    \label{eq:().div.cons}
    (\uIdiv \bvec{v},\uIdiv \bvec{w})_{\DIV,T}&=(\bvec{v},\bvec{w})_{L^2(T)^3}&\qquad&\forall \bvec{v},\bvec{w}\in\Poly{k}(T)^3.
  \end{alignat}
\end{lemma}
\begin{proof}
  Let us first prove the positive definiteness of the bilinear forms $(\cdot,\cdot)_{\bullet,T}$. By inspection, they are positive semi-definite, and it only remains to prove that they are definite. 

  Consider first the case of $(\cdot,\cdot)_{\GRAD,T}$. Let $\underline{q}_T\in\uHgrad$ be such that $(\underline{q}_T,\underline{q}_T)_{\GRAD,T}=0$. Then, obviously from \eqref{eq:().grad.T}, we have $\pgrad\underline{q}_T=0$, $\delta_{\GRAD,T}^{k-1}\underline{q}_T=0$, $\delta_{\GRAD,F}^{k-1}\underline{q}_T=0$ for all $F\in\FT$ and $\delta_{\GRAD,E}^{k+1}\underline{q}_T=0$ for all $E\in\ET$. This gives
  $\underline{0}=\uIgrad(\cancel{\pgrad\underline{q}_T})-\underline{q}_T$, and thus $\underline{q}_T=\underline{0}$ as required.

  The definiteness of $(\cdot,\cdot)_{\CURL,T}$ is obtained exactly the same way, so let us turn to $(\cdot,\cdot)_{\DIV,T}$. If $\uvec{v}_T\in \uHdiv$ is such that $(\uvec{v}_T,\uvec{v}_T)_{\DIV,T}=0$ then $\pdiv\uvec{v}_T=\bvec{0}$ and $\delta_{\DIV,F}^k\uvec{v}_T=0$ for all $F\in\FT$. This shows that $v_F=\lproj{k}{F}(\pdiv\uvec{v}_T\cdot\normal_F)-\delta_{\DIV,F}^k\uvec{v}_T=0$ for all $F\in\FT$. Using then \eqref{eq:pdivT:Goly.k}, we infer that
  \[
  \int_T \cDT\uvec{v}_T~q_T = 0\qquad\forall q_T\in\Poly{0,k}(T).
  \]
  The definition \eqref{eq:DT} of $\cDT$ together with the fact that $\GRAD:\Poly{0,k}(T)\to\Goly{k-1}(T)$ is surjective then shows that $\bvec{v}_{\cvec{G},T}=\Gproj{k-1}{T}\bvec{v}_T=\bvec{0}$.
  Since $\pdiv\uvec{v}_T=\bvec{0}$, the relation \eqref{eq:pdivT:Goly.k.perp} obviously yields $\bvec{v}_{\cvec{G},T}^{\perp}=\bvec{0}$, which concludes the proof of $\uvec{v}_T=\uvec{0}$.

  The consistency properties \eqref{eq:().grad.cons}--\eqref{eq:().div.cons} follow easily from the consistency properties \eqref{eq:pgradT:polynomial.consistency}, \eqref{eq:pcurlT:polynomial.consistency} and \eqref{eq:pdivT:polynomial.consistency} of the potential reconstructions.
\end{proof}


\section*{Acknowledgements}

D. A. Di Pietro acknowledges the partial support of \emph{Agence Nationale de la Recherche} (grants ANR-15-CE40-0005 and ANR-17-CE23-0019).
J. Droniou was partially supported by the Australian Government through the \emph{Australian Research Council}'s Discovery Projects funding scheme (project number DP170100605).
F. Rapetti  warmly thanks the CASTOR team at INRIA Sophia-Antipolis for the delegation in 2019, during which this work was completed.


\begin{footnotesize}
  \bibliographystyle{plain}
  \bibliography{ddr}
\end{footnotesize}

\end{document}